\documentclass[10pt]{amsart}
\usepackage{amsmath,amsthm,amssymb,amsfonts,eucal,amscd,mathbbol,mathrsfs,tikz-cd}
\usepackage[shortlabels]{enumitem}
\usepackage{relsize}
\usepackage{cite}
\usepackage{setspace}
\usepackage[all,cmtip]{xy}
\usepackage{graphicx}
\usepackage{subfig}
\usepackage{leftidx}
\usepackage[bookmarks,bookmarksnumbered, plainpages=false, pdfpagelabels]{hyperref}
\usepackage{xcolor}
\usepackage{extarrows}
\usepackage{rotating}
\usepackage{scalefnt}
\usepackage{enumitem}

\hypersetup{
 colorlinks   = true, 
 urlcolor     = green, 
 linkcolor    = blue, 
 citecolor   = red 
}

\newtheorem{thm}{Theorem}
\newtheorem{cor}{Corollary}
\newtheorem{lem}{Lemma}
\newtheorem{prop}{Proposition}

\theoremstyle{definition}
\newtheorem{defn}{Definition}

\newtheorem{remark}{Remark}

\newtheorem{example}{Example}

\newcommand{\R}{\ensuremath{\mathbb{R}}}

\newcommand{\Z}{\ensuremath{\mathbb{Z}}}

\def\i{\infty}

\def\p{\partial}
\def\cal{\mathcal}

\def\supp{\emph{supp}}

\def\a{\alpha}
\def\b{\beta} 
\def\g{\gamma}
\def\d{\delta}
\def\e{\epsilon}

\def\l{\lambda}
\def\s{\sigma} 

\def\o{\omega}

\def\G{\Gamma}
\def\D{\Delta}

\def\n{\nabla} 

\def\A{\cal{A}}

\def\C{\cal{C}}
\def\E{\cal{E}}

\def\L{\cal{L}}

\def\U{\cal{U}}

\def\fg{\mathfrak{g}}

\DeclareRobustCommand{\svdots}{
  \vcenter{%
    \offinterlineskip
    \hbox{.}
    \vskip0.25\normalbaselineskip
    \hbox{.}
    \vskip0.25\normalbaselineskip
    \hbox{.}%
  }%
}

\renewcommand{\theenumi}{(\alph{enumi})}
\makeatletter
	\renewcommand{\p@enumii}{\theenumi}
\makeatother                                            

\begin{document} 
	
\author[H. Pugh]{H. Pugh} 
\title{Higher Covariant Derivative and the Bundle of Dirac Currents}

\maketitle
\begin{abstract}
Using the higher covariant derivative on a manifold \( M \) equipped with a torsion-free connection, we define a natural surjective bundle map \( \Phi \) from \( (\otimes(TM))\otimes (\wedge(TM)) \) to the vector bundle \( \U(M) \) of de Rham currents on \( M \) supported in a single (variable) point. The resulting quotient bundle can be thought of as a bundle of generalized Weyl algebras, with the symplectic form replaced with the Riemannian curvature tensor. The fibers of the bundle \( \U(M) \) are differential co-algebras, and the boundary, co-product and co-unit stitch together to form bundle maps which lift via \( \Phi \) to commuting bundle maps on \( (\otimes(TM))\otimes (\wedge(TM)) \). Interior product, higher-order covariant differentiation, and their \( L^2 \) adjoints also form bundle maps on \( \U(M) \) which lift via \( \Phi \). The higher-order covariant derivative in particular is an \( \R \)-algebra representation of the space \( C^\i(\otimes(TM)) \) equipped with a non-standard, \emph{covariant product}.  Its composition with interior product yields a quantization of \( \U(M) \) corresponding to a Hopf-algebraic smash product.

Finitely supported and locally finitely supported sections functors can be applied to \( \U(M) \), yielding the spaces of finitely supported and locally finitely supported currents, respectively. In particular, the finitely supported currents on a smooth manifold are a filtered differential graded co-algebra in duality with differential forms.
\end{abstract}

\section{Introduction}
A distribution supported in the origin in \( \R^n \) is a linear combination of the \emph{Dirac delta} distribution and its weak derivatives, \[ f \mapsto \frac{\p^I f}{\p x^I}(0), \] where \( I \) is a multi-index \cite[Ch. III, Thm. XXXV]{schwartz0}. Since differentiation is commutative in \( \R^n \), the vector space of these distributions with values in \( \R \) is isomorphic to the symmetric algebra \( S(T_0\R^n) \). Indeed, a non-zero element \( a^I x_I \) of \( S(T_0 \R^n) \) determines a non-zero compactly supported distribution: Evaluate it against the polynomial \( a_I x^I \).

More generally, a de Rham current of homogeneous dimension \( k \) supported in a point \( p\in\R^n \) is a linear combination of currents \[ \o \mapsto \left(\frac{\p^I}{\p x^I} \o\right)_p(\a_I), \] where \( \a_I\in \wedge^k(T_p\R^n) \) (See e.g. \cite[Lem 2]{melrose}.) The set of such currents is thus isomorphic to \( S(T_p \R^n)\otimes \wedge^k(T_p \R^n) \). Taking the disjoint union over \( p\in \R^n \), the set of real-valued currents supported in at most one point in \( \R^n \) is identified with the bi-graded bundle \( \U(\R^n)=S(T\R^n)\otimes \wedge(T\R^n) \), modulo the zero section (since the zero current is represented as the zero vector in each fiber.)

Replacing \( \R^n \) with a smooth manifold \( M \), it is not hard to show that the set of \emph{Dirac currents} in \( M \), those whose supports are contained in a single point, has the structure of a smooth vector bundle \( \U(M) \) over \( M \), again modulo the zero section. The bundle \( \U(M) \) is filtered by order\footnote{A compactly supported current \( T \) has order \( r \) if it extends to a continuous linear functional on \( C^r \) forms, and does not extend further to \( C^s \) forms for \( s<r \). Every compactly supported current has finite order.} and graded by dimension. We will show:
\begin{thm}
	Each fiber of \( \U(M) \) has the structure of a filtered differential graded co-algebra, where the co-product \( \D \) dualizes to wedge product of differential forms, in the sense that \( \o \otimes \eta (\D T) = (\o \wedge \eta) (T) \), and the differential dualizes to exterior derivative. 
\end{thm}

In the presence of additional structure on \( M \), we will define a natural surjective bundle map \( \Phi: \otimes(TM)\otimes \wedge(TM) \to \U(M) \) so that on fibers \( \Phi \) is a map of filtered and graded co-algebras. This additional structure can be either a Lie group structure, or a torsion-free connection on \( M \). In this paper we will focus on the second case. The first case is addressed in detail in \cite{coprodlie}, but briefly:

\subsection{Lie Groups}
If \( G \) is a Lie group with Lie algebra \( \fg \), we can identify \( \U(G) \) with the (trivial) bundle \( G\times (U(\fg)\otimes \wedge(\fg^-)) \), where \( U(\fg) \) is the universal enveloping algebra of \( \fg \) and \( \fg^- \) is the Lie algebra \( \fg \) with negative Lie bracket. Here we are interpreting the elements of \( U(\fg) \) as left-invariant differential operators \( C^\i(G)\to C^\i(G) \), and \( \wedge^k(\fg^-) \) as right-invariant tangent \( k \)-vector fields, so \( (p, s\otimes \a)\in  G\times U(\fg)\otimes \wedge^k(\fg^-) \) is identified with the current \( \o \mapsto [s(\o)]_p(\a), \) where \( s \) is extended to \( k \)-forms by the Leibniz rule and commutation with \( d \).

The canonical quotient map \( \otimes(\fg )\to U(\fg) \) allows us to write \( \U(G) \) as a quotient of the bundle \( \otimes(TG)\otimes \wedge(TG) \). On fibers, the resulting quotient map is a map of filtered graded co-algebras.	

\subsection{Smooth Manifolds with Connection}
Suppose \( M \) is a smooth manifold equipped with a torsion-free connection \( \n \). In the same way that for a Lie group \( G \), the bundle \( \U_k^r(G) \) is a quotient of \( \otimes^{\leq r} (TG)\otimes \wedge^k(TG) \), we will define a natural surjective bundle map \( \Phi: \otimes^{\leq r}(TM)\otimes \wedge^k(TM) \to \U_k^r(M) \). This map will be defined using the higher covariant derivative operator \( \n^r \) in place of the Lie derivative by left-invariant vector fields.

In Section \ref{sec:covdiv}, we will describe the higher covariant derivative \( \n^\bullet \) and its properties. In particular, we will define a unital, associative \( \R \)-linear \emph{covariant product} \( \odot \) on \( C^\i(\otimes(TM)) \) such that if \( E\to M \) is a bundle with connection, then
\begin{align*}
	\n^\bullet_{\underline{\,\,}}: C^\i(\otimes(TM)) &\to End(C^\i(E)) \\
	X&\mapsto \n^\bullet_X
\end{align*}
is an algebra map, turning \( C^\i(E) \) into a left \( C^\i(\otimes(TM)) \) module. In fact, we show later in Section \ref{section:combined} that the composition of \( \n^\bullet \) with wedge product \( \wedge \),  
\begin{align*} 
\wedge\circ \n^\bullet : C^\i(\wedge(E))\otimes_{C^\i(M)}C^\i(\otimes(TM))&\to End(C^\i(\wedge(E))\\
\eta \otimes X \mapsto \left(\o \mapsto \eta \wedge \n^\bullet_X \o\right),
\end{align*}
is a well-defined algebra map, where the product on the domain takes the from of the Hopf-algebraic smash product. 

In Section \ref{sec:derivativesatapoint}, we will use the higher covariant derivative to define the bundle map \( \Phi \). Analogous to the identification of \( \U_k(G) \) as a quotient of \( \otimes(\fg)\otimes \wedge^k(\fg^-) \) by a sub-bundle generated by terms such as \( (v\otimes w-w\otimes v-[v,w])\otimes \a \), the fiber of \( \U_k(M) \) at \( p \) is identified with a quotient of \( \otimes(T_pM) \otimes \wedge(T_pM) \) by terms such as \( (v\otimes w - w\otimes v)\otimes \a + 1\otimes (R_{v,w} \a) \), where \( v,w\in T_p M \) and \( \a\in \wedge^k(T_pM) \). Thus, the curvature \( R_{v,w} \) is baked into the quotient for \( \U_k(M) \) the same way that the Lie bracket is for \( \U_k(G) \). This is analogous to the defining quotient map of a Weyl algebra, except the symplectic form is vector-valued and the resulting structure is not generally an algebra:
	
We will prove in Theorem \ref{thm:algebra} that unlike \( \U_k(G) \), the fiber of \( \U_k(M) \) at \( p\in M \) does not inherit a product structure from \( \otimes(T_pM) \otimes \wedge(T_pM) \) unless the manifold is \emph{asymptotically flat at \( p \)}, meaning that \( \n^j R_p =0 \) for all \( j\geq 0 \). The fiber of \( \U(M) \) at \( p \) does, however, have a co-product structure \( \D_p \), dual to wedge product on differential forms. As the base point varies, these structures \( \D_p \) form a bundle map \( \D \) defined on \( \U(M) \). Moreover, \( \D \) lifts via \( \Phi \) to a natural bundle map \( \D_\otimes \) defined on \( \otimes(TM) \otimes \wedge(TM) \), which on fibers is the tensor product of the tensor and exterior co-product structures on \( \otimes(T_pM) \) and \( \wedge(T_pM) \), respectively. In particular, on fibers, \( \Phi \) is a map of filtered graded co-algebras.
	
In Section \ref{sec:distinguished}, we will give formulas for certain distinguished endomorphisms of \( \otimes(TM)\otimes \wedge(TM) \) which descend to endomorphisms of \( \U(M) \), namely those dual to interior product and covariant derivative. Using these, we provide a formula for a naturally defined endomorphism of \( \otimes(TM)\otimes \wedge^\bullet(TM) \), which commutes with \( \D_\otimes \) and descends via \( \Phi \) to the boundary operator on the fibers of \( \U(M) \).

Finally, the finitely supported sections functor \( \G^{finite} \) and the locally finitely supported sections functor \( \G^{loc.finite} \) can be applied to \( \U(M) \) to yield the space of finitely supported currents \( \A_\bullet(M) \) on \( M \) and the space of locally finitely supported currents \( \A^{loc}_\bullet(M) \) on \( M \), respectively. The space \( \A(M) \) is a filtered differential graded co-algebra in duality with the differential graded algebra \( \E^\bullet(M) \) of smooth forms on \( M \).

Likewise, \( \A^{loc}_\bullet(M) \) can be put in duality with the space \( \cal{D}^\bullet(M) \), but \( \A^{loc}_\bullet(M) \) is neither filtered nor a co-algebra. With its natural topology it has a \emph{completed} co-algebra structure, however, and the completed co-product of an element is the (possibly infinite) sum of the co-products at each base point. 

\begin{remark}
	The bundle \( \U_k^r(M) \) is naturally isomorphic to \( J^r(\wedge^k T^*M)^* \), the dual bundle of the order-\( r \) jet bundle of \( \wedge^k T^* M \). Its smooth sections \( C^\i(\U_k^r(M)) \) are precisely the order-\( r \) partial differential operators from differential \( k \)-forms on \( M \) to smooth functions on \( M \).

\end{remark}

\subsection{Notation and Terminology}

if \( E\to M\) is a vector bundle, then \( \otimes(E), S(E), \wedge(E), \) and \( End(E) \) are, respectively, the tensor, symmetric, anti-symmetric, and endomorphism functors applied to \( E \). These are vector bundles whose fiber above \( p\in M \) are the tensor algebra \( \otimes(E_p) \), symmetric algebra \( S(E_p) \), exterior algebra \( \wedge(E_p) \), and the space of linear endomorphisms \( End(E_p) \), respectively. To clarify notation, we will often use the square tensor product symbol \( \boxtimes \) to denote the tensor product in various contexts. For example, the tensor product of \( \otimes(TM) \) and \( \wedge(E) \) will be denoted \( \otimes(TM)\boxtimes \wedge(E) \).

We will make use of a number of co-products (as in, the structure pre-dual to an algebraic product,) and will reserve the symbol \( \D \) for such use. On the tensor algebra, the co-product will be assumed to be the one compatible with the tensor product, not the co-free co-product. We will call this the \emph{tensor co-product}. We call the corresponding co-product on the exterior algebra the \emph{wedge co-product}. 

We will also make use of (sumless) Sweedler notation for various co-products: \( \D x = x_{(1)}\otimes x_{(2)}. \) Co-associativity allows us to write \[ x_{(1)}\otimes x_{(2)}\otimes x_{(3)} := (x_{(1)})_{(1)}\otimes (x_{(1)})_{(2)} \otimes x_{(2)} = x_{(1)}\otimes (x_{(2)})_{(1)}\otimes (x_{(2)})_{(2)},\] inductively defining \( x_{(1)}\otimes \cdots \otimes x_{(r)} \) the same way. We will occasionally use Sweedler notation on the indexing set of a basis, the operation being induced by the co-product applied to the underlying basis vectors. 

When each fiber of a vector bundle \( E \) over \( M \) is equipped with a co-product over \( \R \) smoothly varying on the fibers, these co-products can be applied fiber-wise to a smooth section of \( E \) and the resulting section of \( E\otimes E \) can be treated as the tensor product, over \( C^\i(M) \), of sections in \( E \). This defines a co-product on \( C^\i(E) \), as a module over \( C^\i(M) \). We will make particular use of this co-product on the tensor bundle \( \otimes(TM) \), using the tensor co-product at each point.

If \( E \) is an infinite dimensional vector bundle filtered by a family \( (E_i)_{i\in I} \) of finite-dimensional vector sub-bundles \( E_i \) of \( E \), we denote by \( C^\i(E) \) the filtered vector space \( \cup_i C^\i(E_i) \).

\section{Higher-Order Covariant Derivative}\label{sec:covdiv}
Let \( M \) be a smooth manifold and suppose \( E\to M \) and \( F\to M \) are finite dimensional vector bundles with connections \( \n \) and \( \bar{\n} \), respectively. The tensor product connection \( \n\otimes \bar{\n} \) on \( E\otimes F \to M \) is defined via the Leibniz rule as
\begin{equation}\label{connection2}
	\left(\n\otimes \bar{\n}\right)_X(\a\otimes\b)= \left(\n_X\a\right)\otimes \b + \a \otimes \left(\bar{\n}_X \b\right).
\end{equation}

The dual connection \( \n^* \) on \( E^*\to M \) is defined so that the covariant derivative commutes with contraction \( E\otimes E^* \to M\times \R \), i.e. 
\begin{equation}\label{connection1}
	\left(\n^*_X \o\right)(Y) := X(\o(Y))-\o\left(\n_X Y\right).
\end{equation}

Suppose the tangent bundle \( TM \to M \) is also equipped with a connection, which we will denote \( \hat{\n} \). Putting \eqref{connection1} and \eqref{connection2} together, we get a connection \( \tilde{\n} \) on \( T^* M \otimes E \):
\[ 
\tilde{\n}_X(\o \otimes \b) = \hat{\n}^*_X(\o)\otimes \b + \o\otimes\n_X \b \quad \in C^\i\left(T^*M\otimes E\right).
\]
Contracting this with a vector field \( Y \) on \( M \), 
\begin{align*}
	\left(\tilde{\n}_X(\o\otimes\b)\right)(Y)&=\left( \hat{\n}^*_X(\o) \right)(Y)\cdot\b + \o(Y)\cdot \n_X\b \quad \in C^\i(E)\\
	&= X(\o(Y))\cdot \b - \o\left( \hat{\n}_X Y\right) \cdot \b + \o(Y) \cdot \n_X\b\\
	&=\n_X(\o(Y)\cdot\b) - \o\left(\hat{\n}_X Y\right) \cdot \b.
\end{align*}
If we replace \( \o\otimes\b \) with \( \n \a \), where \( \a \) is a section of \( E \), we see that the map 
\[ 
	\tilde{\n}\circ\n : C^\i(E)\to C^\i\left(T^*M \otimes \ T^*M\otimes E\right)
\] 
is given by the formula
\begin{align*}
	\left(\tilde{\n}\circ\n (\a)\right)(X,Y) &= \n_X \n_Y \a - \n_{\hat{\n}_X Y}\a\\
	&=: \n_{X,Y}^2 \a,
\end{align*}
i.e. the second covariant derivative of \( \a \). Observe that \( \n_{X,Y}^2 \a \) is tensorial in \( X \) and \( Y \), and that if \( \hat{\n} \) is torsion-free, then the commutator \( \left( \n_{X,Y}^2 - \n_{Y,X}^2 \right) \a \) is tensorial in \( \a \). In this case, we let \[ R_{X,Y}^E = \n_{X,Y}^2 - \n_{Y,X}^2, \] and call \( R_{X,Y}^E \) the Riemann curvature tensor on \( E \). To simplify notation, we will reuse the superscript \( E \) when \( E \) replaced with a tensor, exterior, or symmetric power thereof.

We can extend the above process indefinitely:

\begin{tikzcd}
	C^\i(E) \arrow[r, "\n"] & C^\i\left(T^*M\otimes E\right) \arrow[r, "\tilde{\n}"] & C^\i\left(T^*M\otimes T^*M\otimes E\right) \arrow[r, "\tilde{\tilde{\n}}"] & C^\i\left(T^*M^{\otimes 3}\otimes E\right) \arrow[r] & \cdots
\end{tikzcd}

\vphantom{$\tilde{X}$}

Define
\[
\n^j = \underbrace{\tilde{\overset{\text{\smaller[2]$\svdots$}}{\tilde{\n}}}\,\,\circ\cdots\circ\tilde{\tilde{\n}}\circ\tilde{\n}\circ\n}_{j\text{-times}} \  :C^\i(E)\to C^\i\left(T^*M ^{\otimes j} \otimes E\right).
\]

By contracting with \( j \) tangent vector fields, this gives us an order-\( j \) higher covariant derivative
\[
\n_{X_1,\dots,X_j}^j: C^\i(E)\to C^\i(E),
\]
tensorial in \( X_1,\dots, X_j \). Alternatively, we will use the notation \( \n_{X_1\otimes\dots\otimes X_j}^j \) to denote this map, where the tensor is taken over \( C^\i(M) \). It will also be convenient to write \( \n_1^0=Id \), and in general \( \n_f^0 \a = f\cdot\a \). To avoid visual clutter, we will occasionally omit the superscript, so that \( \n_v \) refers to the order-\( j \) covariant derivative if \( v\in C^\i(\otimes^j(TM)) \).

The formula for \( \n^j \) is given inductively:
\begin{equation}\label{eq:higher}
	\n_{X_0,\dots,X_j}^{j+1} = \n_{X_0}\circ\n_{X_1,\dots,X_j}^{j} - \sum_{k=1}^j \n_{X_1,\dots,X_{k-1},\hat{\n}_{X_0}X_k,X_{k+1},\dots,X_j}^{j}.
\end{equation}

\begin{example}
	\label{ex:3rd}
	\( \n_{X,Y,Z}^3= \n_X\circ\n_{Y,Z}^2-\n_{\hat{\n}_X Y, Z}^2 - \n_{Y,\hat{\n}_X Z}^2 \).
\end{example}

In fact, using Sweedler notation for the tensor co-product, 
\begin{prop}\label{prop:mega}
	For \( v\in C^\i(\otimes^i (TM) )\) and \( w\in C^\i(\otimes^j (TM)) \), where \( i,j\geq 0 \),
	\begin{equation}\label{eq:mega}
		\n_v^i \circ \n_w^j  = \n_{v_{(1)}\otimes\hat{\n}_{v_{(2)}}^\bullet w}^{i+j-\bullet},
	\end{equation}
	
	where \( 0\leq \bullet \leq i \) is the integer such that \( v_{(2)}\in C^\i(\otimes^\bullet TM) \).
\end{prop}

\begin{proof}
	We proceed by induction on \( i \), proving the \( i=0 \) and \( i=1 \) cases separately: The \( i=0 \) case expresses the tensoriality of \( \n^j \), which can be verified from the definition, or from \eqref{eq:higher}, which is the \( i=1 \) case:
	
	Suppose \( \a\in C^\i(E) \) and write \( \n^j \a \) as an Einstein summation of simple tensors, \( \n^j \a = \o_1^s\otimes \cdots \otimes \o_j^s \otimes \g^s \). Then
	\begin{align*}
		\n_{X_0,\dots,X_j}^{j+1}\a &= \tilde{\overset{\text{\smaller[2]$\svdots$}}{\tilde{\n}}}\,\, \left(\o_1^s\otimes \cdots \otimes \o_j^s \otimes \g^s \right) \left(X_0,\dots,X_j\right)\\
		&=\left[ \left( \n^* \o_1^s \right) \otimes \o_2^s \otimes \cdots \otimes \o_j^s \otimes \g^s + \cdots + \o_1^s\otimes \cdots\otimes \o_{j-1}^s\otimes \left(\n^* \o_j^s\right) \otimes \g^s \right.\\
			&\qquad\qquad \left. + \o_1^s\otimes\cdots\otimes\o_j^s\otimes\left( \n \g^s \right) \right] \left(X_0,\dots,X_j\right)\\
		&= \left(\n_{X_0}^* \o_1^s \right)\left(X_1\right)\o_2^s\left(X_2\right)\cdots \o_j^s\left(X_j\right)\g^s + \cdots + \o_1^s\left(X_1\right)\cdots\o_{j-1}^s \left(X_{j-1}\right)\left( \n_{X_0}^* \o_j^s \right)\left(X_j\right)\g^s\\
			&\qquad\qquad  + \o_1^s\left(X_1\right)\cdots \o_j^s\left(X_j\right)\n_{X_0}\g^s\\
		&=\n_{X_0}\left[ \o_1^s\left(X_1\right)\cdots \o_j^s\left(X_j\right)\g^s \right]\\
			&\qquad\qquad  - \o_1^s\left(\hat{\n}_{X_0}X_1\right)\o_2^s\left(X_2\right)\cdots \o_j^s\left(X_j\right)\g^s - \cdots - \o_1^s\left(X_1\right)\cdots \o_{j-1}^s\left(X_{j-1}\right)\o_j^s\left( \hat{\n}_{X_0} X_j \right)\g^s\\
		&= \left[\n_{X_0}\n_{X_1,\dots,X_j}^j  - \n_{\hat{\n}_{X_0}X_1,X_2,\dots,X_j}^j - \dots - \n_{X_0,\dots,X_{j-1},\hat{\n}_{X_0}X_j}^j\right] \a.
	\end{align*}
	
	Now supposing \eqref{eq:mega} holds for \( i \), let us show it holds for \( i+1 \). Assume WLOG that \( v=v_0\otimes\cdots\otimes v_i \), and let \( u=v_1\otimes\cdots\otimes v_i \), so that \( v=v_0 \otimes u \). Then
	
	\begin{align*}
		\n_v^{i+1} \circ \n_w^j &= \n_{v_0} \circ \n_{u}^{i} \circ \n_w^j - \n_{\hat{\n}_{v_0}u}^{i} \circ \n_w^j\\
		&= \n_{v_0} \circ \n^{i+j-\bullet}_{u_{(1)} \otimes \hat{\n}^\bullet_{u_{(2)}}w} - \n^{i+j-\bullet}_{\left(\hat{\n}_{v_0}u\right)_{(1)}\otimes\n^\bullet_{\left(\hat{\n}_{v_0}u\right)_{(2)}}w}\\
		&= \n^{i+j+1-\bullet}_{v_0\otimes u_{(1)}\otimes \hat{\n}^\bullet_{u_{(2)}}w} + \n^{i+j-\bullet}_{\left(\hat{\n}_{v_0}\left(u_{(1)}\right)\right)\otimes \hat{\n}^\bullet_{u_{(2)}}w} + \n^{i+j-\bullet}_{u_{(1)}\otimes \hat{\n}_{v_0}\hat{\n}^\bullet_{u_{(2)}}w} - \n^{i+j-\bullet}_{\left(\hat{\n}_{v_0}u\right)_{(1)}\otimes\hat{\n}^\bullet_{\left(\hat{\n}_{v_0}u\right)_{(2)}}w}\\
		&= \n^{i+j+1-\bullet}_{v_0\otimes u_{(1)}\otimes \hat{\n}^\bullet_{u_{(2)}}w} + \n^{i+j-\bullet}_{\left(\hat{\n}_{v_0}\left(u_{(1)}\right)\right)\otimes \hat{\n}^\bullet_{u_{(2)}}w} +  \n^{i+j+1-\bullet}_{u_{(1)}\otimes \hat{\n}^\bullet_{v_0\otimes u_{(2)}}w} + \n^{i+j-\bullet}_{u_{(1)}\otimes \hat{\n}^\bullet_{\hat{\n}_{v_0} u_{(2)}}w}\\ &\qquad\qquad -\n^{i+j-\bullet}_{\left(\hat{\n}_{v_0}u\right)_{(1)}\otimes\hat{\n}^\bullet_{\left(\hat{\n}_{v_0}u\right)_{(2)}}w}.
	\end{align*}
	The first and third term sum to \( \n^{i+j+1-\bullet}_{v_{(1)}\otimes \hat{\n}^\bullet_{v_{(2)}}w} \) and the other terms cancel, leaving the desired equality.
\end{proof}

A similar calculation shows that the higher covariant derivatives satisfy the Leibniz rule: 
\begin{prop}\label{prop:leibniz}If \( \a\in C^\i(E) \) and \( \b\in C^\i(F) \), and \( v\in C^\i(\otimes^j (TM) )\) then
	\[
	(\n\otimes \bar{\n})_v^j (\a \otimes_{C^\i(M)} \b) = (\n_{v_{(1)}}^{j-\bullet}\a) \otimes_{C^\i(M)} (\bar{\n}_{v_{(2)}}^\bullet \b).
	\]
\end{prop}
Proposition \ref{prop:leibniz} holds for the tensor bundle \( \otimes(E) \), as well as the exterior \( \wedge(E) \) and symmetric \( S(E) \) bundles, given that the covariant derivative \( \n \) on \( E \) is extended to these bundles using the corresponding Leibniz rule \eqref{connection2}. Let us assume this is the case for the rest of the paper. Explicitly in the case of \( \wedge(T^*M) \), letting \( \n=\hat{\n}^* \),
\begin{prop}\label{prop4}
	If \( \o \) and \( \eta \) are differential forms on \( M \), then\footnote{This also holds for iterated Lie Derivatives: \[ \L_{v_1}\cdots \L_{v_j} (\o \wedge \eta)= \sum_{i=0}^j \sum_{\s\in Sh(i,j-i)}\left(\L_{v_{\sigma(1)}}\cdots\L_{v_{\sigma(i)}} \o \right)\wedge \left(\L_{v_{\sigma(i+1)}}\cdots\L_{v_{\sigma(j)}} \eta\right). \]
} 
	\[
	\n_{v_1,\dots,v_j}^j (\o \wedge \eta) = \sum_{i=0}^j \sum_{\s\in Sh(i,j-i)}\n_{v_{\sigma(1)},\dots,v_{\sigma(i)}}^{i}\o \wedge \n_{v_{\sigma(i+1)},\dots,v_{\sigma(j)}}^{j-i} \eta,
	\]
	where \( Sh(p,q)\subset S_{p+q} \) is the set of \( (p,q) \) riffle shuffle permutations. 
\end{prop}
 
To simplify notation from here on out, let us re-use the notation \( \n \) for the dual connection on \( E^* \), as well as the connection on \( F \), and the tensor product connection on the Hom bundle \( Hom(E,F)\simeq E^*\otimes F  \). However, let's still keep track of the connection on \( TM \) separately, using \( \hat{\n} \). Since the covariant derivative commutes with contraction, so too do the higher covariant derivatives, and Proposition \ref{prop:leibniz} implies:
\begin{prop}
	\label{prop:contract}
	If \( \a\in C^\i(E) \) and \( \omega \in C^\i(E^*) \), or more generally if \( \omega\in C^\i(Hom(E,F)) \), then 
 	\[
 	\n_v^j\left(\omega(\a)\right)=\left(\n_{v_{(1)}}^{j-\bullet}\omega\right)\left( \n_{v_{(2)}}^\bullet \a \right).
 	\]
\end{prop}
 
Induction on \( j \) and Proposition \ref{prop:contract} imply:
 
\begin{cor}\label{cor:pullback}
	Let \( T\in C^\i(Hom(E,F)) \). Given a third finite dimensional vector bundle \( G\to M \) with connection, the pullback section \( T^*\in C^\i(Hom(Hom(F,G), Hom(E,G))) \) satisfies \[ \left[ \left(\n_v^j (T^*)\right)\omega \right](\a) =  \omega\left(\left(\n_v^j T \right) (\a)\right) \] for all \( \a\in C^\i(E) \) and \( \omega\in C^\i(Hom(F,G)) \). I.e., \[ \n_v^j (T^*)=\left(\n_v^j T\right)^*. \]
	
	Dually, the pushforward section \( T_*\in C^\i(Hom(Hom(G,E), Hom(G,F))) \) satisfies \[ \n_v^j (T_*)= \left(\n_v^j T\right)_*. \]
\end{cor}

We will apply this later in the case that \( T \) is the Riemann curvature tensor and \( G=M\times \R \).

\begin{proof}
	The \( j=0 \) case is vacuous. Assuming the statement is true for all \( \n_u^i \) with \( i<j \),  we have by Proposition \ref{prop:contract},
	\begin{align*}
		\omega\left(\left(\n_v^j T\right)\a\right)&=\n_v^j (\omega(T\a)) - \left(\n_{v_{(1)}}\omega\right)\left(\left(\n_{v_{(2)}}^{<j}T\right)\left(\n_{v_{(3)}} \a \right)\right) \\
		&=\n_v^j \left(T^*\omega(\a)\right) - \left(\left(\n_{v_{(2)}}^{<j}\left(T^*\right)\right)\left(\n_{v_{(1)}}\omega\right)\right)\left(\n_{v_{(3)}} \a \right)\\
		&=\left[\left(\n_v^j \left(T^*\right)\right)\omega\right] (\a).
	\end{align*}
	
	The case of pushforward is similar.
\end{proof}

\begin{cor}\label{cor:consistency}
	If \( X\in C^\i(\otimes^i(E)) \), then for any \( k\geq 0 \), \[ \n_v^j(X\otimes \cdot)=(\n_v^j X)\otimes \cdot, \] where \( X\otimes \cdot \in C^\i(Hom(\otimes^k(E), \otimes^{k+i}(E))) \) is the map \( Y\mapsto X\otimes Y \). 
\end{cor}

\begin{proof}
	Let us denote by \( M_X \) the operator \( Y\mapsto X\otimes Y \). By Proposition \ref{prop:contract}, \( \n_v^j(X\otimes Y)=(\n_{v_{(1)}}^\bullet(M_X))(\n_{v_{(2)}}^{j-\bullet}(Y)) \). On the other hand, by Proposition \ref{prop:leibniz}, \( \n_v^j(X\otimes Y) = \n_{v_{(1)}}^\bullet(X)\otimes \n_{v_{(2)}}^{j-\bullet}(Y) = M_{\n_{v_{(1)}}^\bullet(X)}(\n_{v_{(2)}}^{j-\bullet}(Y)) \), so that \[ (\n_{v_{(1)}}^\bullet(M_X))(\n_{v_{(2)}}^{j-\bullet}(Y)) = M_{\n_{v_{(1)}}^\bullet(X)}(\n_{v_{(2)}}^{j-\bullet}(Y)) \] The result follows from induction on \( j \).
\end{proof}

Corollary \ref{cor:consistency} also holds for \( \wedge(E) \) and \( S(E) \). In particular, since interior product \( \iota_X \) is dual to \( Y\mapsto X\wedge Y \), it follows from Corollary \ref{cor:pullback} that

\begin{cor}\label{cor:interior}
 	If \( X\in C^\i(\wedge(E)) \) and \( \o\in C^\i(Hom(\wedge(E), F)) \), then \[ \n_v^j \left(\iota_X \o\right)= \iota_{\n_{v_{(1)}}^\bullet X} \left(\n_{v_{(2)}}^{j-\bullet} \o\right). \]
\end{cor}

The higher covariant derivative commutes with the (\( C^\i(M) \)-linear) co-product on the tensor, exterior, and symmetric algebras:
\begin{prop}\label{prop:covcoprod}
	If \( \a \) is a smooth section of either \( \otimes^\bullet(E) \), \( \wedge^\bullet(E) \) or \( S^\bullet(E) \), then
	\[
	\n_v^j(\D \a) = \D \n_v^j \a.
	\]
\end{prop}
 
\subsection{Commutation}
 
Just as the commutation relation \( \n_{X,Y}-\n_{Y,X} \) is a tensor (given that \( \hat{\n} \) is torsion-free,) so too are commutations of the even-order higher covariant derivatives:

\begin{prop}
	If the underlying connection \( \hat{\n} \) on \( TM \) is torsion-free, then \( \n^{2j}_{(v_1 w_1 - w_1 v_1)\cdots(v_j w_j - w_j v_j)} \) is a tensor for all \( j \). That is, if \( f\in C^\i(M) \) and \( \a \) is a section of \( E \), then 
	\[ 
	\n^{2j}_{(v_1 w_1 - w_1 v_1)\cdots(v_j w_j - w_j v_j)} (f\cdot \a) = f\cdot \n^{2j}_{(v_1 w_1 - w_1 v_1)\cdots(v_j w_j - w_j v_j)} \a.
	\]
Additionally, the value of the tensor \( \n^{2j}_{(v_1 w_1 - w_1 v_1)\cdots(v_j w_j - w_j v_j)} \) at \( p \) is completely determined by the value of the Riemann curvature tensors \( R^{TM} \) and \( R^E \) at \( p \).
\end{prop}


\begin{proof}
	The proof is by induction on \( j \). The \( j=1 \) case being evident, we have by Proposition \ref{prop:mega},
	\[
	\n^{2(j+1)}_{(v_0 w_0 - w_0 v_0)\cdots(v_j w_j - w_j v_j)} = R_{v_0,w_0}^E \circ \n^{2j}_{(v_1 w_1 - w_1 v_1)\cdots(v_j w_j - w_j v_j)} - \n^{2j}_{R_{v_0,w_0}^{TM}\left[(v_1 w_1 - w_1 v_1)\cdots(v_j w_j - w_j v_j)\right]}.
	\]
	Since the Riemann curvature tensor \( R^{TM} \) extends to \( \otimes(TM) \) as a derivation, the second term on the right hand side is composed of terms that satisfy the inductive hypothesis, and the first term is similarly a tensor and a function of the Riemann curvature tensors \( R^{TM} \) and \( R^E \) by induction. 
\end{proof}

\begin{example}
	Consider \( j=2 \). The preceding proof shows \[ \n^4_{(a b - b a) (c d - d c)} = R_{a,b}^E \circ R_{c,d}^E - R^E_{R^{TM}_{a,b}(c\otimes d)}. \]
\end{example}

We have the following general formula for commutation of two individual factors:

\begin{lem}[Fundamental Commutation Lemma]\label{lem:fundamental}
	Suppose the underlying connection \( \hat{\n} \) on \( TM \) is torsion-free. If \( i,j\geq 0 \), \( u\in C^\i(\otimes^i (TM)) \), \( v\in C^\i(\otimes^j (TM) ) \), and \( a,b\in C^\i(TM) \), then
	\begin{equation}\label{eq:fundamental}
		\n_{u_{(1)}\hat{\n}_{u_{(2)}}((ab-ba)v)} = \left(\n_{u_{(1)}} R^E \right)_{\hat{\n}_{u_{(2)}}(ab)}\circ\n_{u_{(3)}\hat{\n}_{u_{(4)}} v} - \n_{u_{(1)}\left(\hat{\n}_{u_{(2)}} R^{TM}\right)_{\hat{\n}_{u_{(3)}}(ab)}(\hat{\n}_{u_{(4)}}v)}.
	\end{equation}
\end{lem}

\begin{proof}
	By Propositions \ref{prop:mega}, \ref{prop:leibniz}, and \ref{prop:contract},
		\begin{align*}
			\n_{u_{(1)}\hat{\n}_{u_{(2)}}((ab-ba)v)} &= \n_u \circ \n_{(ab-ba)v}\\
			&=\n_u\circ (R^E_{a,b}\circ \n_v - \n_{R^{TM}_{a,b}(v)}) \\
			&=\left(\n_{u_{(1)}} (R^E_{a,b}) \right)\circ\n_{u_{(2)}}\circ \n_v - \n_{u_{(1)}\hat{\n}_{u_{(2)}}\left(R^{TM}_{ab}v\right)}\\
			&=\left(\n_{u_{(1)}} R^E \right)_{\hat{\n}_{u_{(2)}}(ab)}\circ\n_{u_{(3)}\hat{\n}_{u_{(4)}}v} - \n_{u_{(1)}\left(\hat{\n}_{u_{(2)}} R^{TM}\right)_{\hat{\n}_{u_{(3)}}(ab)}(\hat{\n}_{u_{(4)}}v)}.
		\end{align*}
\end{proof}
The left hand side of \eqref{eq:fundamental} consists of \( \n_{u(ab-ba)v}^{i+2+j} \) and higher covariant derivatives of order \( < i+2+j \), and by moving these lower-order terms to the right hand side, this gives a formula for \( \n_{u(ab-ba)v}^{i+2+j} \) in terms of lower-order covariant derivatives and the Riemannian curvature tensors \( R^E \) and \( R^{TM} \). Note that since \( \n_{u(ab-ba)v}^{i+2+j} \) is tensorial in \( u \), \( v \), \( a \) and \( b \), therefore so too must be the sum of the rest of the terms in \eqref{eq:fundamental}. 

\begin{remark}
	A word of caution: The trivial bundle \( M\times \R \) has a unique covariant derivative, given by \( \n_X f = X(f) \). The second covariant derivative satisfies \( \n^2_{X,Y}f = \n^2_{Y,X}f \) if the underlying connection \( \hat{\n} \) on \( TM \) is torsion-free. That is, \( R^{M\times \R}=0 \). However, by Lemma \ref{lem:fundamental}, \[ \n^3_{X,Y,Z}f-\n^3_{Y,X,Z}f = -(R^{TM}_{X,Y}Z)(f), \] which is not generally zero.
\end{remark}

\subsection{Covariant Product}

The composition formula in Proposition \ref{prop:mega} motivates the definition of a \emph{covariant product} on\footnote{By \( C^\i(\otimes(TM)) \) we mean the space of possibly non-homogeneous but globally finite-order smooth contravariant tensor fields on \( M \).} \( C^\i(\otimes(TM)) \):
\begin{defn}
	Let \( \odot \) be the map
	\begin{align*}
		\odot: C^\i(\otimes(TM))\otimes_\R C^\i(\otimes(TM)) &\to C^\i(\otimes(TM))\\
		a\otimes_{\R} b \mapsto a_{(1)}\otimes_{C^\i(M)} \n_{a_{(2)}} b.
	\end{align*}
	where \( a_{(1)}\otimes_{C^\i(M)} \n_{a_{(2)}} b \) is identified with an element of \( C^\i(\otimes(TM)) \) by the natural isomorphism \[ C^\i(X\otimes Y)\simeq C^\i(X)\otimes_{C^\i(M)} C^\i(Y) \] for vector bundles \( X \) and \( Y \) over \( M \).
\end{defn}

\begin{example}
	Suppose \( \hat{\n} \) is torsion-free. Given smooth vector fields \( V \) and \( W \) on \( M \), \[ V\odot W - W\odot V = V\otimes W - W\otimes V  + [V,W]. \]
\end{example}

Proposition \ref{prop:covcoprod} implies that \( \odot \) turns \( A:=C^\i(\otimes(TM)) \) into a unital, associative filtered \( \R \)-algebra.

Given a bundle \( E\to M \) with connection \( \n \), Proposition \ref{prop:mega} implies that the higher covariant derivative \( \n^\bullet \) on \( E \) is a filtered algebra map from \( A \) to the filtered algebra of differential operators \( C^\i(E)\to C^\i(E) \). In particular,
\begin{prop}
	The space \( C^\i(E) \) is a left module over \( A \).
\end{prop}

Proposition \ref{prop:covcoprod} also implies that \( \odot \) is compatible with the tensor co-product \( \D \), in the sense that the diagram
	\[\begin{tikzcd}
		{A\otimes A} & A & {A\otimes A} \\
		{A\otimes A\otimes A \otimes A} && {A\otimes A\otimes A\otimes A}
		\arrow["\odot", from=1-1, to=1-2]
		\arrow["{\D \otimes \D}"', from=1-1, to=2-1]
		\arrow["\D", from=1-2, to=1-3]
		\arrow["{id\otimes \tau \otimes id}"', from=2-1, to=2-3]
		\arrow["{\odot \otimes \odot}"', from=2-3, to=1-3]
	\end{tikzcd}\]
	commutes. This nearly turns \( A \) into a bi-algebra, except \( \D \) acts over \( C^\i(M) \) and \( \odot \) is only \( C^\i(M) \)-linear in its first component. 
	
	Moreover, if \( E \) is any one of \( \otimes(F) \), \( \wedge(F) \) or \( S(F) \) for a smooth vector bundle \( F\to M \), then \( \n_v 1 = \e(v) 1 \) and \( \n_v(\a\b) = (\n_{v_{(1)}}\a)(\n_{v_{(2)}}\b) \) (by Proposition \ref{prop:leibniz},) for \( \a,\b\in C^\i(E) \) and \( v\in A \). Thus, if \( A \) were in fact a Hopf algebra, then \( C^\i(E) \) would be a (left) Hopf \( A \)-module algebra. Nevertheless, we will show later in Section \ref{section:combined} that the above properties are enough to define a \( C^\i(M) \) version of the Hopf-algebraic smash product \( \C^\i(E)\sharp A \). That is,
	\begin{thm}
		The space \( B:=A\otimes_{C^\i(M)} C^\i(\wedge(F)) \), equipped with the product \[ (v\otimes \a) \sharp (w\otimes \b) =  (w_{(1)}\odot v)\otimes (\n_{w_{(2)}}\a)\wedge \b, \] is a unital associative \( \R \)-algebra, and the space \( C^\i(\wedge(F)) \) is a left module over \( B \), where the action is given by composition of wedge product and higher covariant derivative, \( \wedge\circ \n^\bullet \).
\end{thm}


\subsection{Christoffel symbols}
Higher-order Christoffel symbols can be defined as follows: Let \( (e^i)_{i=1}^n \) be coordinate functions on a neighborhood in \( M \). Let \( I \) be a word with letters in \( \{1,\dots, n\} \), and let \( e_I \) be the coordinate tangent tensor field corresponding to \( I \).
\begin{defn}
	The higher-order Christoffel symbols \( \G_{\ I,j}^k \) for \( j, k\in \{1,\dots, n\} \) are functions defined by the formula 
	\begin{equation}\label{eq:christoffel}
		\hat{\n}_{e_I}^{|I|}e_j=\G_{\ I,j}^k e_k.
	\end{equation}
\end{defn}
Using \eqref{eq:higher}, we have the following inductive formula for the higher-order Christoffel symbols:
\begin{equation}
	\G_{\ I,j}^k=\frac{\p \G_{\ I',j}^k}{\p e^{i_1}}+\G_{\ I',j}^l \cdot\G_{\ i_1,l}^k - \G_{\ i_1,i_2}^l\cdot\G_{\ I_l^1,j}^k - \cdots - \G_{\ i_1,i_s}^l\cdot\G_{\ I_l^s,j}^k,
\end{equation}
where \( I=(i_1,\dots,i_s), \) \( I'=(i_2,\dots,i_s) \), and \( I_l^r=(i_2,\dots,l,\dots,i_s) \), where the \( l \) occurs in the \( r \)-th positon. 

In particular, if \( \hat{\n} \) is torsion-free, then \( R^k_{\ juv}=\G^k_{\ uv,j}-\G^k_{\ vu,j} \).

The following two lemmas will be useful later, and since the second involves the higher-order Christoffel symbols and builds on the first, we include them here:

\begin{lem}
	\label{lem:monomial}
	Let \( p\in M \) and let \( (e^i)_{i=1}^n \) be coordinate functions on a neighborhood of \( p \). Suppose \( S \) is a word with letters in \( \{1,\dots n\} \) and \( T \) is an \( n \)-dimensional multi-index. If \( |S|\leq |T| \), then
	\begin{equation*}
		\left[ \hat{\n}_{e_S}^{|S|}\frac{(e-p)^T}{T!} \right]_p = \d_{\langle S\rangle,T},
	\end{equation*}
	where \( (e-p)^T \) is the monomial centered at \( p \) with variables in \( (e^i)_{i=1}^n \) whose degrees are determined by \( T \), the quantity \( T! \) is the product of the factorials of the components of \( T \), the symbol \( \d \) is the Kronecker delta, and \( \langle S\rangle \) is the \( n \)-dimensional multi-index counting the number of occurrences of each letter in \( S \). The quantities \( |S| \) and \( |T| \) are the length of the word \( S \) and the sum of the entries of \( T \), respectively (so that \( |\langle S\rangle|=|S| \).)
	
	More generally, if in addition \( \a\in C^\i(E) \), then 
	\begin{equation*}
		\left[ \n_{e_S}^{|S|}\frac{(e-p)^T\cdot \a}{T!} \right]_p = \d_{\langle{S}\rangle,T}\cdot \a_p.
	\end{equation*}
\end{lem}

\begin{proof}
	This follows from Proposition \ref{prop:leibniz} and the pigeonhole principle, treating the monomial \( (e-p)^T \) as the tensor product of \( |T| \) separate functions.
\end{proof}

The assumption that \( |S|\leq |T| \) can be dropped if the higher-order Christoffel symbols are zero at \( p \):

\begin{lem}
	\label{lem:euclidean}
	Let \( (e^i)_{i=1}^n \) be coordinates about \( p\in M \) and suppose \( \G_{\ I,j}^k(p)=0 \) for all \( |I|\geq 1 \), \( j,k\in \{1,\dots,n\} \). If \( s\geq 0 \), then
	\begin{enumerate}
		\item\label{item:zero} \( \left[\hat{\n}^s\G_{\ I,j}^k\right]_p=0 \) for all \( |I|\geq 1 \), \( j,k\in \{1,\dots,n\} \); and
		\item\label{item:nr} \( \left[\hat{\n}^s R^{TM}\right]_p=0 \).
	\end{enumerate}
	Furthermore, if \( S=i_1\cdots i_s \) is a word with letters in \( \{1,\dots, n\} \), then:
	\begin{enumerate}
		\setcounter{enumi}{2}
		\item\label{item:comp} \( \left[\n_{e_S}^s \a\right]_p = \left[\n_{e_{i_1}}\circ\dots\circ \n_{e_{i_s}}\a\right]_p \) for all \( \a\in C^\i(E) \); and
		\item\label{item:monomial} If \( T \) is an \( n \)-dimensional multi-index, then (in the notation of Lemma \ref{lem:monomial},) \[ \left[ \hat{\n}_{e_S}^{s}\frac{(e-p)^T}{T!} \right]_p = \d_{\langle{S}\rangle,T}. \] More generally, if in addition \( \a\in C^\i(E) \), then 
	\begin{equation*}
		\left[ \n_{e_S}^{s}\frac{(e-p)^T\cdot \a}{T!} \right]_p = \d_{\langle{S_{(1)}}\rangle,T}\cdot \left(\n_{e_{S_{(2)}}}^{\left|S_{(2)}\right|}\a\right)_p.
	\end{equation*}
	\end{enumerate}
\end{lem}

\begin{proof}
	We prove \ref{item:zero} and \ref{item:comp} first:
	\begin{enumerate}
		\item[\ref{item:zero}] Applying \( \hat{\n}_{e_S}^s \) to both sides of \eqref{eq:christoffel} and using Propositions \ref{prop:mega} and \ref{prop:leibniz} yields \[ 0=\left[\hat{\n}_{e_{SI}}^{s+|I|} e_j\right]_p = \left[\left(\hat{\n}_{e_S}^s\G_{\ I,j}^k\right) e_k\right]_p. \] The result follows by linear independence of the \( e_k \)'s.
		\item[\ref{item:comp}] We proceed by induction on \( s \): The claim is vacuous for \( s=1 \). Suppose we have shown \( \left[\n_{e_T}^t \b\right]_p = \left[\n_{e_{i_1}}\circ\dots\circ \n_{e_{i_t}}\b\right]_p \) for all \( \b\in C^\i(E) \) and \( |T|=t<s \). It follows from Proposition \ref{prop:mega} and our assumption that the Christoffel symbols vanish at \( p \) that 
			\begin{equation*}
				\left[\n_{e_S}^s \a\right]_p = \left[\n_{e_{i_1}}\circ\n_{e_{i_{2}},\dots,e_{i_s}}^{s-1}\a\right]_p.
			\end{equation*}
			Now suppose we have shown
			\begin{equation*}
				\left[\n_{e_S}^s \a\right]_p = \left[\n_{e_{i_1}}\circ\dots\circ \n_{e_{i_{j-1}}}\circ\n_{e_{i_{j}},\dots,e_{i_s}}^{s-j+1}\a\right]_p.
			\end{equation*}
			
			By Proposition \ref{prop:mega},
			\begin{equation}\label{eq:ind0}
				\left[\n_{e_S}^s \a\right]_p = \left[\n_{e_{i_1}}\dots \n_{e_{i_{j-1}}} \left( \n_{e_{i_j}} \n_{e_{i_{j+1}},\dots,e_{i_s}}^{s-j} - \n_{\n_{e_{i_j}}e_{i_{j+1}},\dots,e_{i_s}}^{s-j} \cdots - \n_{e_{i_{j+1}},\dots,\n_{e_{i_j}}e_{i_s}}^{s-j} \right)\a\right]_p.
			\end{equation}
			Using our inductive step with \( t=j-1 \), \eqref{eq:ind0} is equal to
			\begin{equation*}
				\left[\n_{e_{i_1}}\dots\n_{e_{i_{j}}} \n_{e_{i_{j+1}},\dots,e_{i_s}}^{s-j}\a - \n_{e_{i_1},\dots, e_{i_{j-1}}}^{j-1}\left(\G_{\ i_j, i_{j+1}}^l \n^{s-j}_{e_l,\dots,e_{i_s}} \a +\cdots + \G_{\ i_j, i_{s}}^l \n^{s-j}_{e_{i_{j+1}},\dots,e_{l}} \a \right)\right]_p,
			\end{equation*}
			all but the first term of which vanishes by Proposition \ref{prop:leibniz} and \ref{item:zero}. The result follows by induction on \( j \).
		\item[\ref{item:monomial}] By \ref{item:comp}, the higher covariant derivative of the monomial can be computed as a sequence of directional derivatives and the result is immediate.
		\item[\ref{item:nr}] The claim follows from Proposition \ref{prop:contract} and \ref{item:comp} using induction on \( s \). In the case that \( s=0 \), we have for \( v\in T_p M \),
			 \begin{align*}
			 	R_{e_i, e_j}^{TM} v &= (\hat{\n}_{e_i}\circ\hat{\n}_{e_j}-\hat{\n}_{e_j}\circ\hat{\n}_{e_i}) (v^k e_k)\\
				&=((\hat{\n}_{e_i}\circ\hat{\n}_{e_j}-\hat{\n}_{e_j}\circ\hat{\n}_{e_i})v^k )e_k\\
				&=([e_i,e_j] v^k) e_k\\
				&=0.
			 \end{align*}
			 Now suppose we have shown \( (\hat{\n}^r R^{TM})_p=0 \) for all \( r<s \). By Proposition \ref{prop:contract} and our inductive step,
			 \begin{equation}\label{eq:euc0}
			 	(\hat{\n}^s_S R^{TM})_{e_i,e_j} (v) = \hat{\n}^s_S(R_{e_i,e_j}^{TM}(v^k e_k)).
			 \end{equation} 
			 By Proposition \ref{prop:mega}, the right hand side of \eqref{eq:euc0} is equal to \[ (\hat{\n}^{s+2}_{Se_ie_j}-\hat{\n}^{s+2}_{Se_je_i})(v^k e_k), \]
			 which by Proposition \ref{prop:leibniz} and our assumption on the higher-order Christoffel symbols is equal to \[ (\hat{\n}^{s+2}_{Se_ie_j}-\hat{\n}^{s+2}_{Se_je_i})(v^k) e_k. \] By \ref{item:comp}, this is equal to \[ (\hat{\n}^s_S\circ R_{e_i,e_j}^{M\times \R} (v^k)) e_k = 0. \] 
	\end{enumerate}
\end{proof}

\section{Derivatives at a Point}\label{sec:derivativesatapoint}
\begin{defn}
	Suppose \( M \) is a smooth \( n \)-dimensional manifold equipped with a torsion-free connection \( \hat{\n} \) and that \( \pi: E\to M \) is a \( d \)-dimensional vector bundle with connection \( \n \). Fix \( p\in M \) and let\footnote{The dependence on \( \n \) and \( \hat{\n} \) will occasionally be made explicit since later we will be varying these.} \( \Phi^k_p = \Phi^k_p(E,\hat{\n},\n) \) be the linear map\footnote{The \( \boxtimes \) symbol will be used throughout to denote this external tensor product in place of \( \otimes \) for notational clarity.},
	\begin{align}\label{eq:Phip}
		\Phi_p^k:  \otimes(T_pM) \boxtimes \wedge^k(E_p) &\rightarrow \left[C^\i(\wedge^k(E^*))\right]^*\\
		(v_1\otimes\cdots\otimes v_j)\boxtimes \a &\mapsto \left[ \o \mapsto \left(\n_{\tilde{v}_1\otimes\cdots\otimes \tilde{v}_j}^j \o \right)_p(\a)\right]\notag,
	\end{align}
	where the pairing between \( \wedge^k(E_p) \) and \( \wedge^k(E_p^*) \) is given by the determinant of the pairing between \( E \) and \( E^* \), and \( \tilde{u}\in C^\i(TM) \) refers to a smooth extension of \( u\in T_pM \). Since \( \n^j \) is tensorial in its \( j \) vector field arguments and we are immediately evaluating the resulting section \( \n_{\tilde{v}_1\otimes\cdots\otimes \tilde{v}_j}^j \o \) at \( p \), the definition of \( \Phi_p^k \) is independent of the choice of extensions \( \tilde{v}_i \) of \( v_i \).
\end{defn}

\begin{example}\label{ex:kernel0}
	Let \( a,b\in T_pM \) and \( \a\in \wedge^k(E_p) \). Then \[ \Phi_p^k((a\otimes b - b\otimes a)\boxtimes \a + 1\boxtimes R_{a,b}^E\a)=0. \]
\end{example}

\begin{defn}
	Let \( \Phi_p^{r,k} \) denote the restriction of \( \Phi_p^k \) to \( \otimes^{\leq r}(T_pM) \boxtimes \wedge^k(E_p) \) and let \( \U_k^r(E;p) \) be the image of \( \Phi_p^{r,k} \). Also, in a slight abuse of notation, let \( \U_k^r(M;p) \) denote the space \( \U_k^r(TM;p) \).
\end{defn}

The space \( \U_k^r(E;p) \) is independent of \( \hat{\n} \) and \( \n \), and is contained in the continuous dual space of \( C^\i(\wedge^k(E^*)) \), given its natural Frechet topology. We are primarily interested in the cases that \( E=TM \) and \( E=N \) where \( N \) is the normal bundle of an embedding of a Riemannian manifold \( M \) in \( \R^n \). In the first case,

\begin{prop}\label{prop:surjective}
	The space \( \U_k^r(M;p) \) is the space of order \( \leq r \) de Rham \( k \)-currents\footnote{A compactly supported current \( T \) has order \( r \) if it extends to a continuous linear functional on \( C^r \) forms, and does not extend further to \( C^s \) forms for \( s<r \). Every compactly supported current has finite order.} on \( M \) supported in \( \{p\} \).
\end{prop}

\begin{defn}\label{def:flat}
	We say \( M \) (resp. \( E \)) is \emph{asymptotically flat} at \( p \) if \( \hat{\n}^j(R^{TM})(p)=0 \) (resp. \( \n^j(R^E)(p)=0 \)) for all \( j\geq 0 \). By Lemma \ref{lem:euclidean}, \( M \) is asymptotically flat if the higher-order Christoffel symbols are zero in some coordinate system about \( p \).
\end{defn}

The domain of \( \Phi_p \) is an algebra, being the tensor product of tensor and exterior algebras, but \( \Phi_p \) only preserve this algebra structure if both \( M \) and \( E \) are asymptotically flat at \( p \). In fact,
\begin{thm}
	\label{thm:algebra}
	The following are equivalent:
	\begin{enumerate}
		\item \( ker(\Phi_p) \) is an ideal;\label{equiv:1}
		\item \( M \) and \( E \) are asymptotically flat at \( p \);\label{equiv:2}
		\item The map \( \Phi_p \) factors through \( S(T_pM) \boxtimes \wedge(E_p) \).\label{equiv:3}
	\end{enumerate}
	If any of the above hold, then \( \Phi_p \) descends to an isomorphism\footnote{In any case, there is an isomorphism between the two spaces depending on a trivialization of \( E \) in a coordinate chart about \( p \). See also Corollary \ref{cor:unnatural}.} between \(  S(T_p M) \boxtimes \wedge^\bullet (E_p) \) and \( \U_\bullet (E;p) \).
\end{thm} 

To prove Theorem \ref{thm:algebra}, we will make use of the following lemma, which implies the equivalence between \ref{equiv:2} and \ref{equiv:3}:

\begin{lem}
	\label{lem:algebra}
	Let \( K\geq 0 \). The tensors \( \hat{\n}^l R^{TM} \) and \( \n^l R^E \) are zero at \( p \) for all \( l\leq K \) if and only if for all \( i\leq K \), \( u\in C^\i(\otimes^i (TM)) \), \( j\geq 0 \), \( v\in C^\i(\otimes^j (TM) ) \), \( a,b\in C^\i(TM) \), and \( \o\in C^\i(\wedge(E)) \),
	\begin{equation}\label{eq:exchange}
	\n^{i+2+j}_{u(a b-b a) v}\o(p)=0.
	\end{equation}
\end{lem}

\begin{proof}
	We proceed by induction on \( K \). First assume \( K=0 \). By Proposition \ref{prop:mega}, 
	\begin{equation}\label{eq:lemalg0}
		\n^{2+j}_{(a b-b a) v} = R_{a,b}^E\circ \n_v^j - \n_{R_{a,b}^{TM}(v)}^j,
	\end{equation}
	from which necessity is clear. For sufficiency, let \( v=1 \) (i.e. the \( (j=0)\)-tensor field equal to \( 1 \).) Then \( R_{a,b}^{TM}(v)=0 \) and \( \n_v^0=Id \), showing \( R_{a,b}^E=0 \) at \( p \). Now let \( j=1 \). By \eqref{eq:lemalg0}, \( \n_{R_{a,b}^{TM}(v)} \o (p) = 0 \) for all \( \o\in C^\i(\wedge(E)) \). If \( x:=R_{a,b}^{TM}(v)(p)\neq 0 \) for some choice of \( a,b  \) and \( v \), then choosing \( \o \) equal to the monomial determined by Lemma \ref{lem:monomial} with \( S=x \) yields a contradiction. Thus, \( R^{TM}(p)=0 \).
	
	Now assume the statement holds for \( K=L-1 \), we will show it holds for \( K=L \). To prove necessity, suppose \( \hat{\n}^l R^{TM} \) and \( \n^l R^E \) vanish at \( p \) for all \( l\leq L \). It suffices by induction to show \eqref{eq:exchange} for \( i=L \). In this case, Lemma \ref{lem:fundamental} states
	\begin{equation}\label{eq:last}
		\n^\bullet_{u_{(1)}\hat{\n}_{u_{(2)}}((ab-ba)v)} =\left(\n_{u_{(1)}}^{L-\star-\bullet} R^E \right)_{\hat{\n}_{u_{(2)}}^\star(ab)}\circ\n_{u_{(3)}}^\bullet\circ \n_v^j - \n^\bullet_{u_{(1)}\left(\hat{\n}_{u_{(2)}} R^{TM}\right)_{\hat{\n}_{u_{(3)}}(ab)}(\hat{\n}_{u_{(4)}}v)}.
	\end{equation}
	All the terms of the right hand side of \eqref{eq:last}, when evaluated on \( \o \), are zero at \( p \), and the left hand side of \eqref{eq:last} consists of the sum of \( \n^{i+2+j}_{u(a b-b a) v} \o (p) \) and terms which are zero at \( p \) by the inductive hypothesis. This proves necessity.
	
	For sufficiency, suppose \eqref{eq:exchange} holds for all \( i\leq L \). By the inductive hypothesis, \( \hat{\n}^l R^{TM} \) and \( \n^l R^E \) both vanish at \( p \) for all \( l<L \). We need to show that \( \hat{\n}^L R^{TM} \) and \( \n^L R^E \) also vanish at \( p \). Lemma \ref{lem:fundamental} implies
	\begin{equation}\label{eq:first}
		0 = \left(\left(\n_u^L R^E \right)_{ab}\circ \n_v^j\o - \n^\bullet_{\left(\hat{\n}_u R^{TM}\right)_{ab}(v)}\o\right)(p).
	\end{equation}
	The proof is concluded exactly as the \( K=0 \) case.
\end{proof}

\begin{proof}[Proof of Theorem \ref{thm:algebra}]
	By Lemma \ref{lem:algebra}, \ref{equiv:2} and \ref{equiv:3} are equivalent. To see that \ref{equiv:1} implies \ref{equiv:3}, suppose \( ker(\Phi_p) \) is an ideal and let \( x, y \in T_p M\). Then \( (x\otimes y - y\otimes x )\boxtimes 1 \in ker(\Phi_p), \) so \( ker(\Phi_p) \) contains the kernel of the natural projection \(  \otimes(T_pM) \boxtimes \wedge(E_p) \to  S(T_pM) \boxtimes \wedge(E_p). \)
	
	Finally, we show that \ref{equiv:3} implies that \( [\Phi_p] \) is injective as a map from \( S(T_pM) \boxtimes \wedge(E_p) \) to \( C^\i(\wedge(E^*))^* \), which in turn implies \ref{equiv:1}.
	
	Pick local coordinates \( (e^i)_{i=1}^n \) about \( p \) and let \( s \) be a non-zero element of \( S^j (T_pM) \boxtimes \wedge^k (E_p) \). Write \( s \) in terms of these coordinates, say, \( s=a^{I,J}\cdot e_I \boxtimes \a_J \), and let \( e_{\bar{I}}\boxtimes \a_{\bar{J}} \) be a term in the sum so that \( |\bar{I}| \) is maximized. Let \( \o=\frac{1}{\bar{I}!}(e-p)^{\bar{I}}\cdot \o \), where \( \o\in C^\i(\wedge(E^*)) \) is any section such that \( \o_p(\a_{\bar{J}})\neq 0 \). Then by Lemma \ref{lem:monomial}, \( \Phi_p(s)(\o)\neq 0 \).
\end{proof}

On the other hand, the domain of \( \Phi_p \) is also a co-algebra since it is the tensor product of co-algebras, and
\begin{thm}
	\label{thm:co-alg}
	The space \( \U_\bullet(E;p) \) has a filtered and graded co-associative co-algebra structure via duality with wedge product on \( C^\i(\wedge(E^*)) \), and the map \( \Phi_p \) is a filtered and graded co-algebra homomorphism.
\end{thm}
In particular, \( ker(\Phi_p) \) is a co-ideal.
\begin{proof}
	This follows in its entirety from Proposition \ref{prop4}. 
	
	Let \( \o\in C^\i(\wedge^k(E^*)) \), \( \eta\in C^\i(\wedge^j(E^*)) \), and \( x\in \otimes^\star(T_pM) \boxtimes \wedge^l (E_p) \), where \( l=k+j \). Writing \( x=s_i\boxtimes \a_i \), we have \( \D(x)=s_{i(1)}\boxtimes \a_{i(1)}\otimes {s}_{i(2)}\boxtimes {\a}_{i(2)} \). Then by Proposition \ref{prop4},
	\begin{align}
		\Phi_p(x)(\o\wedge\eta)&=\left(\n_{\tilde{s}_i}(\o\wedge\eta)\right)_p(\a_i)\label{eq:thefirst}\\
		&=\left(\n_{\tilde{s}_{i(1)}}\o\right)\wedge \left(\n_{\tilde{s}_{i(2)}}\eta\right)_p(\a_i)\notag\\
		&=\left(\n_{\tilde{s}_{i(1)}}\o\right)_p(\a_{i(1)})\cdot\left(\n_{\tilde{s}_{i(2)}}\eta\right)_p(\a_{i(2)})\notag\\
		&=\Phi_p(s_{i(1)}\boxtimes \a_{i(1)})(\o) \cdot \Phi_p(s_{i(2)}\boxtimes \a_{i(2)})(\eta)\notag\\
		&=\Phi_p\otimes \Phi_p (\D(x))(\o\otimes \eta)\label{eq:thelast}.
	\end{align}
	The left hand side of \eqref{eq:thefirst} is \( \wedge^*(\Phi_p(x))(\o\otimes \eta) \), where \( \wedge^* \) is the dual map to the exterior product \[ \wedge: \oplus_{a+b=l}C^\i(\wedge^a(E^*))\otimes C^\i(\wedge^b(E^*)) \to C^\i(\wedge^l(E^*)). \]
	By \eqref{eq:thefirst}-\eqref{eq:thelast}, the restriction of \( \wedge^* \) to \( \U_l(E;p) \) factors through \[ \oplus_{a+b=l} \U_a(E;p) \otimes \U_b(E;p) \hookrightarrow (\oplus_{a+b=l}C^\i(\wedge^k(E^*))\otimes C^\i(\wedge^b(E^*)))^* \] 
	and thus defines a co-associative co-product \( \D_p: \U_{l}(E;p)\to \oplus_{a+b=l}\U_a(E;p)\otimes \U_b(E;p) \), satisfying
	\begin{equation}\label{eq:coalghomo}
		\D_p(\Phi_p(x))=\Phi_p\otimes \Phi_p (\D(x)).
	\end{equation}
	We also have a co-unit \( \e_p: \U(E;p)\to \R \) given by \( \e_p(T)=T(1) \), where \( 1\in C^\i(\wedge^0(E^*)) \) is the function on \( M \) identically equal to \( 1 \), i.e. the unit for wedge product. Commutativity of the co-unit diagram similarly follows from \eqref{eq:thefirst}-\eqref{eq:thelast}. Moreover,
	\[ \e_p(\Phi_p(x))= \Phi_p(s_i\otimes \a_i)(1) = (\n_{\tilde{s}_i}1)_p (\a_i), \]
	the right hand of which is zero except for \( s_i \in \otimes^0(T_pM) \) and \( \a_i\in \wedge^0(E_p) \), in which case it is equal to \( s_i\cdot \a_i \). In other words, \( \e_p\circ \Phi_p \) is the co-unit for \( \otimes(T_pM)\boxtimes \wedge(E_p) \), and this together with \eqref{eq:coalghomo} implies that \( \Phi_p \) is a co-algebra homomorphism.
\end{proof}

\begin{cor}
	\label{cor:co-alg}
	The co-product structure on \( \U(E;p) \) (hence that on \( \G^{finite}(\U(E)) \)) is intrinsic to the smooth structure of \( M \), and does not depend on the connections \( \hat{\n} \) or \( \n \).
\end{cor}
\begin{proof}
	Given two torsion-free connections, \( \hat{\n} \) and \( \hat{\n}' \) on \( M \) and connections \( \n \) and \( \n' \) on \( E \), let \( \D_p \) and \( \D_p' \) be the co-algebra structures resulting from Theorem \ref{thm:co-alg}. If \( y=(\D_p -\D_p') (x)\neq 0 \) for some \( x\in \U(E,p) \), write \( y=\sum_{i=1}^N y_1^i\otimes y_2^i \), with \( (y_1^i)_{i=1}^N \) and \( (y_2^i)_{i=1}^N \) linearly independent. In particular, by Hahn-Banach, we can find \( \o \) and \( \eta \in C^\i(\wedge(E^*)) \), such that \( \o \) is non-zero on \( y_1^1 \) but vanishes on \( y_1^i \) for \( i>1 \), and \( \eta \) is non-zero on \( y_2^1 \) but vanishes on \( y_2^i \) for \( i>1 \). Thus, \( (\o \otimes \eta) (y)\neq 0 \), but on the other hand \[ (\o \otimes \eta) (y)= (\o \otimes \eta) ( \D_p x) - (\o \otimes \eta) (\D_p x), \] and by Theorem \ref{thm:co-alg}, both the terms on the right hand side are equal to \( (\o\wedge \eta) (x) \), a contradiction.
\end{proof}

We will conclude this section with two Poincar\'e-Birkhoff-Witt-type theorems. The first, Theorem \ref{thm:PBW1} will give us an explicit basis of \( ker(\Phi_p) \), and the second, Theorem \ref{thm:PBW2}, an explicit basis of \( \U(E;p) \). We will show in the next section that \( \sqcup_{p\in M} \U(E;p) \) has the structure of a smooth filtered and graded vector bundle, and using Theorem \ref{thm:PBW1}, we will be able to conclude that if \( M \) and \( E \) are orientable, then this filtration is by oriented vector bundles.

We first turn our attention to the kernel of \( \Phi_p \). By Lemma \ref{lem:fundamental} and Corollary \ref{cor:pullback} we have the following twisted commutation relation:
\begin{lem}\label{lem:modifiedsweedler}
	If \( i, j\geq 0 \), \( u\in \otimes^i(T_pM) \), \( v\in \otimes^j(T_pM) \), \( a,b\in T_pM \), and \( \a\in \wedge(E_p) \), then
	\begin{align}
	u(ab-ba)v \boxtimes \a &+ u_{(1')}\otimes \left[\n_{\tilde{u}_{(2')}}((\tilde{a}\tilde{b}-\tilde{b}\tilde{a})\tilde{v})\right]_p\boxtimes \a +\notag \\ 
	&u_{(1)}\otimes\left[\n_{\tilde{u}_{(2)}} \tilde{v}\right]_p \boxtimes \left(\n_{\tilde{u}_{(3)}}R^E\right)_{\n_{\tilde{u}_{(4)}}(\tilde{a}\tilde{b})}(\a) +
	u_{(1)}\otimes\left(\n_{\tilde{u}_{(2)}} R^{TM}\right)_{\n_{\tilde{u}_{(3)}}(\tilde{a}\tilde{b})}\left(\n_{\tilde{u}_{(4)}} \tilde{v}\right)_p\boxtimes \a\label{twisted}
	\end{align}
	is in the kernel of \( \Phi_p \), where \( u_{(1')} \) and \( u_{(2')} \) are the Sweedler notation factors of \( \D u -  u\otimes 1\) (i.e. all the factors of \( \D u \) except for \( u\otimes 1 \).) 
\end{lem}
Here, the precise element \eqref{twisted} depends on the choice of extensions \( \tilde{a}, \tilde{b} \) and \( \tilde{v} \), but not the extension \( \tilde{u} \) of \( u \). Any two such elements differ by an element of the kernel of \( \Phi_p \) of tensor order \( < i + j + 2 \).

Let \( \cal{N}\times \R^d \to \pi^{-1}(\cal{N}) \) be a trivialization of \( E \) in a coordinate neighborhood \( \cal{N} \) of \( p \), let \( e^i:\cal{N}\to \R \), \( i=1,\dots,n \) be coordinates for \( M \) in \( \cal{N} \), let \( e_i\in C^\i(T\cal{N}) \) be the corresponding coordinate tangent vector fields, let \( e_I\in C^\i(\otimes(T\cal{N})) \) be the tensor field corresponding to a word \( I \) with letters in \( \{1,\dots,n\} \), and let \( \e_K \) be the coordinate \( k \)-vector field in \( \wedge^k(\pi^{-1}(\cal{N})) \) corresponding to an order-\( k \) subset \( K \) of \( \{1,\dots,d\} \).
\begin{defn}\label{def:kerbasis}
	Given words \( I \) and \( J \) with letters in \( \{1,\dots,n\} \), \( 1\leq i,j\leq n  \), and a subset \( K \) of \( \{1,\dots,d\} \), let \( E_{I,i,j,J,K} \) be the element of \( \otimes(T_pM) \boxtimes \wedge^k(E_p) \) given by
	\begin{align*}
		\left(e_{I_{(1)}}\right)_p\otimes \left(\left(\n_{e_{I_{(2)}}}e_i\right)_p\otimes \left(\n_{e_{I_{(3)}}}e_j\right)_p - \left(\n_{e_{I_{(2)}}}e_j\right)_p\otimes \left(\n_{e_{I_{(3)}}}e_i\right)_p\right)\otimes \left(\n_{e_{I_{(4)}}}e_J\right)_p \boxtimes \left(\e_K\right)_p &+\\
		\left(e_{I_{(1)}}\right)_p\otimes \left(\n_{e_{I_{(2)}}} e_J \right)_p \boxtimes \left(\n_{e_{I_{(3)}}}R^E\right)_{\left(\n_{e_{I_{(4)}}}e_i\right)\otimes\left(\n_{e_{I_{(5)}}}e_j\right)}\left(\e_K\right)_p &+\\
		\left(e_{I_{(1)}}\right)_p\otimes\left(\n_{e_{I_{(2)}}} R^{TM}\right)_{\left(\n_{e_{I_{(3)}}}e_i\right)\otimes\left(\n_{e_{I_{(4)}}}e_j\right)}\left(\n_{e_{I_{(5)}}} e_J\right)_p\boxtimes \left(\e_K\right)_p&.
	\end{align*}
\end{defn} 

\begin{thm}\label{thm:PBW1}
	Let \( r\geq 0 \) and \( 0\leq k\leq d \). The set \( B \) of elements \( E_{I,i,j,J,K} \) where \( 1\leq i<j\leq n \), the words \( I \) and \( J \) range over the words in \( \{1,\dots,n\} \) such that \( |I| + |J| + 2 \leq r \), and \( K \) ranges over the order-\( k \) subsets of \( \{1,\dots,d\} \), forms a basis of the kernel of \( \Phi_p^{r,k} \).
	
	Moreover, this basis has an ordering given by the lexicographical ordering of the words \( IijJK \).
\end{thm}

\begin{proof}
	By Lemma \ref{lem:modifiedsweedler} and our assumption that \( |I| + |J| + 2 \leq r \), the set \( B \) is a subset of the kernel of \( \Phi_p^{r,k} \). The set \( B \) is linearly independent by bootstrapping: Given a linear relation \( \cal{R} \) given by \[ \l_{I_1,i_1,j_1,J_1,K_1} E_{I_1,i_1,j_1,J_1,K_1} + \cdots + \l_{I_s,i_s,j_s,J_s,K_s} E_{I_s,i_s,j_s,J_s,K_s}=0, \] each element \( E_{I_h,i_h,j_h,J_h,K_h} \) in which \( |I_h|+|J_h| \) is maximized within the relation \( \cal{R} \) has a term of maximal tensor order within \( \cal{R} \), namely
	\begin{align*}
		(e_{I_h})_p\otimes \left((e_{i_h})_p\otimes(e_{j_h})_p - (e_{j_h})_p\otimes(e_{i_h})_p\right)\otimes (e_{J_h})_p \boxtimes (\e_{K_h})_p,
	\end{align*}
	occurring when \( (e_{I_h})_{(1)}=e_{I_h} \). Moreover, these terms are \emph{all} such maximal tensor order terms in \( \cal{R} \). Therefore, \( \cal{R} \) must induce a relation between them. However, such elements are part a basis of the kernel of the natural projection \( \otimes^{\leq r}(T_pM) \to S^{\leq r}(T_pM) \), and so in particular are linearly independent. Thus, the coefficients \( \l_{I_h,i_h,j_h,J_h,K_h} \) are zero, and we may repeat the process, given that \( |I_h|+|J_h| \) is now strictly smaller.
	
	The set \( B \) spans by dimension counting: we know the quotient of \( \otimes^{\leq r}(T_pM) \boxtimes \wedge^k(E_p) \) by the kernel of \( \Phi_p^{r,k} \) is isomorphic to \( \U_k^r(E;p) \), and the latter is isomorphic to \( S^{\leq r}(\R^n) \boxtimes \wedge^k (\R^d) \), via the trivialization \( \pi^{-1}(\cal{N})\simeq \R^n\times \R^d \).
\end{proof}

By Theorem \ref{thm:PBW1}, the basis elements \( E_{I,i,j,J,K} \) transform in the following manner under a change-of-basis to new coordinates \( \bar{e}^1,\dots,\bar{e}^n \) and new trivialization \( \bar{\e} \) of \( E \):
\begin{align}\label{eq:changeofbasis}
	E_{I,i,j,J,K} &= \left[ \frac{\p \bar{e}^L}{\p e^I} \cdot\left( \n_{\bar{e}_{L_{(1)}}}\left(\frac{\p\bar{e}^s}{\p e^i}\right)\cdot\n_{\bar{e}_{L_{(2)}}}\left(\frac{\p\bar{e}^t}{\p e^j}\right) \right. \right. \\
	 & \qquad\qquad \left. \left. - \n_{\bar{e}_{L_{(1)}}}\left(\frac{\p\bar{e}^t}{\p e^i}\right)\cdot\n_{\bar{e}_{L_{(2)}}}\left(\frac{\p\bar{e}^s}{\p e^j}\right) \right) \cdot\n_{\bar{e}_{L_{(3)}}}\left(\frac{\p\bar{e}^M}{\p e^J}\right)\cdot\frac{\p\bar{\e}^N}{\p \e^K} \right]_p \bar{E}_{L_{(4)},s,t,M,N},\nonumber
\end{align}
where \( \bar{E}_{L_{(4)},s,t,M,N} \) refers to the basis element corresponding to the new coordinates and the Einstein sum is taken with the restriction that \( s<t \). This follows from Proposition \ref{prop:leibniz} and rearranging terms. 

We also have the following Poincar\'e-Birkhoff-Witt theorem for \( \U_k^r(E;p) \). Let \( \cal{N}_n^r \) denote the set of words \( W=w_1\cdots w_r \) of length \( \leq r \) with letters \( w_i \) in \( \{1,\dots,n\} \) such that \( w_i\leq w_j \) if \( i\leq j \).

\begin{thm}\label{thm:PBW2}
	The set \( B \) of linear functionals on \( C^\i(\wedge^k (E^*)) \) of the form
	\[
	\o \mapsto (\n_{e_I}\o)_p (e_K),
	\]
	where \( I \) ranges over \( \cal{N}_n^r \) and \( K \) ranges over the order \( k \) subsets of \( \{1,\dots, d\} \), is a basis of \( \U_k^r(E;p) \). Moreover, \( B \) is ordered by the lexicographical order of \( IK \). 
\end{thm}

\begin{proof}
	The set \( B \) spans \( \U_k^r(E;p) \) by Lemma \ref{lem:fundamental} and Proposition \ref{prop:surjective}, and is linearly independent by dimension counting, since \( \U_k^r(E;p) \) is isomorphic to \( S^{\leq r}(\R^n) \boxtimes \wedge^k (\R^d) \), and the latter space has a canonical basis with the same cardinality as \( B \).
\end{proof}

Theorem \ref{thm:PBW2} gives a co-algebra isomorphism, depending on a choice of ordered basis of \( T_pM \) and torsion-free connection \( \n \), between \( S^{\leq r} (T_p M) \boxtimes \wedge^k (E_p) \) and \( \U_k^r(E;p) \), since the co-product preserves lexicographic order. If \( M \) and \( E \) are asymptotically flat at \( p \), then this isomorphism is independent of the basis, and is given by \( \left[ \Phi_p^{r,k} \right] : S^{\leq r} (\R^n) \boxtimes \wedge^k (\R^d) \to \U_k^r(E;p)  \).

\section{Bundle Structure of \( \U \)}
To recap, for each \( p\in M \) we have a co-algebra structure \( (\D_p,\e_p) \) on \( \U_\bullet^r(E;p) \) intrinsic to the bundle \( E \) and a surjective co-algebra homomorphism 
\[ \Phi_p: \otimes(T_pM)\boxtimes \wedge(E_p) \to \U(E;p), \] which depends on a connection \( \n \) on \( E \) and a torsion-free connection \( \hat{\n} \) on \( TM \). We begin the process of stitching together this pointwise data into bundles and bundle maps.

\begin{defn}
	Let \( \R^{n,d} \) denote the trivial bundle \( \R^n\times \R^d \to \R^n \). 
\end{defn}

\begin{prop}\label{prop:vb}
	The set \( \U_k^r(E):=\bigsqcup_{p\in M} \U_k^r(E;p) \) has the structure of a smooth vector bundle over \( M \), filtered\footnote{The filtration on \( \U_\bullet^\star(E) \) by \( \star \) does not induce a \( \Z/2 \) grading in general.} by \( r \) and graded by \( k \), with fiber \( \U_k^r(E;p) \) isomorphic to \( S^{\leq r}(\R^n) \boxtimes \wedge^k(\R^d) \).
\end{prop}

\begin{defn}
	We denote the bundle \( \U_\bullet^\star(TM) \) by \( \U_\bullet^\star(M) \).
\end{defn}

\begin{proof}[Proof of Proposition \ref{prop:vb}]
	Given an open cover of \( M \) by coordinate charts \( (\psi_\a:\cal{N}_\a \to \R^n)_{\a\in \A} \) together with trivializations \( (\eta_\a: \cal{N}_\a \times \R^d \to \pi^{-1}(\cal{N}_\a))_{\a\in \A} \) of \( E \), we specify transition data \( (g_{\b\a}: \cal{N}_\a\cap \cal{N}_\b \to Gl(S^{\leq r}(\R^n) \boxtimes \wedge^k(\R^d)))_{\a,\b\in \A} \) as follows: 
	
	For \( \a\in \A \), let \( \xi_\a: \R^{n,d} \to \pi^{-1}(\cal{N}_\a) \) be the map \( \eta_\a \circ (\psi_\a^{-1} \times Id) \). This is a bundle map covering \( \psi_\a^{-1} \).
	
	For \( \a\in \A \) and \( p\in \cal{N}_\a \), let \( G_\a(p) \) be the composition of maps 
	\begin{equation*}
		S^{\leq r}(\R^n) \boxtimes \wedge^k(\R^d) \xlongrightarrow{} S^{\leq r} (T_{\psi_\a(p)}\R^n) \boxtimes \wedge^k(\R^{n,d}_{\psi_\a(p)}) \xlongrightarrow{F_\a(p)} \U_k^r(\R^{n,d}; \psi_\a(p)) \xlongrightarrow{(\xi_\a)_*} \U_k^r(E;p)
	\end{equation*}
	where the first map is induced by the trivial identifications \( \R^n \simeq T_{\psi_\a(p)}\R^n \) and \( \R^d \simeq \R^{n,d}_{\psi_\a(p)} \), the second map \( F_\a(p):=[\Phi_{\psi_\a(p)}^{r,k}(\R^{n,d},\hat{\n}^{\R^n},\n^{\R^{n,d}})] \), where \( \hat{\n}^{\R^n} \) is the Euclidean connection on \( \R^n \), \( \n^{\R^{n,d}} \) is the Euclidean connection on \( \R^{n,d} \), and the third map is pushforward\footnote{This is defined over a general bundle \( E \) the same way as it is over \( TM \), i.e. for a current: Given a bundle map \( \theta: E\to F \), for \( x\in \U_k(E;p) \) and \( \o\in C^\i(F^*) \), define \( \theta_* x (\o):= x(\theta^* \o) \), where \( \theta^*(\o)(\a):= \o(\wedge^k(\theta)(\a)) \). } by \( \xi_\a \).
	
	Given \( \a,\b\in A \) and a point \( p\in \cal{N}_\a\cap \cal{N}_\b \), let \( g_{\b\a}(p)=  G_\b(p)^{-1} \circ G_\a(p). \) This transition data trivially satisfies the \v{C}ech cocycle condition and does not depend on the choice of cover of \( M \) or trivializations of \( E \). Smoothness of the map \( g_{\b\a} \) follows from explicit calculation of the matrix coefficients for \( g_{\b\a} \), which we undertake now. 
	
	Suppose that the \( \cal{N}_\a \) (resp. \( \cal{N}_\b \)) coordinates are labeled \( (x^i)_{i=1}^n \) (resp. \( (y^i)_{i=1}^n \)) and that the coordinates for \( E \) over \( \cal{N}_\a \) (resp. \( \cal{N}_\b \)) induced by \( \eta_\a \) (resp. \( \eta_\b \)) are labeled \( (e^i)_{i=1}^d \) (resp. \( (f^i)_{i=1}^d \).) Fix a basis element \( x_I\boxtimes e_K \) of \( S^{\leq r}(\R^n) \boxtimes \wedge^\bullet(\R^d) \). Our goal is to prove that the coefficients \( a^{J,L} \) such that
	\begin{equation}\label{eq:transitioncoeffs}
		a^{J,L}\cdot G_p^r(\b)(y_J\boxtimes f_L)= G_p^r(\a)(x_I\boxtimes e_K)
	\end{equation} 
	are smooth. Both sides of \eqref{eq:transitioncoeffs} can be evaluated against an element \( \o \) of \( C^\i(\wedge^k (E^*)) \), and this evaluation depends only on the value of \( \o \) near \( p \). Write \( \o=h_M df^M \) near \( p \), where \( df^i \) is the covector field dual to \( f_i \). Fix a multi-index \( A \) in \( n \) variables and an anti-symmetric index \( B \) in \( d \) variables. We have \[ G_p^r(\b)(y_A\boxtimes f_B)(\o) = \left[\frac{\p^{|A|}}{\p y^A} h_B\right]_p. \] On the other hand, \( \o=h_M \frac{\p f^M}{\p e^W} de^W \), and
	\[
	G_p^r(\a)(x_I\boxtimes e_K)(\o) = \left[\frac{\p^{|I|}}{\p x^I}\left(h_M \frac{\p f^M}{\p e^K}\right)\right]_p,
	\]
	so \eqref{eq:transitioncoeffs} becomes
	\begin{equation}\label{eq:basis1}
		a^{A,B}\cdot\left[\frac{\p^{|A|}}{\p y^A} h_B\right]_p = \left[\frac{\p^{|I|}}{\p x^I}\left( h_M \frac{\p f^M}{\p e^K}\right)\right]_p.
	\end{equation}
	
	Let \( \o=\frac{1}{J!}(y-p)^J \cdot df^L \). Here, \( (y-p)^J \) is the monomial centered at \( p \) with indeterminants in \( y \) from the multi-index \( J \), and \( J! \) is the coefficient resulting from differentiating the monomial \( |J| \) times with respect to these indeterminants. For this choice of form \( \o \), \eqref{eq:basis1} reduces by Lemma \ref{lem:euclidean} to
	\begin{equation}\label{eq:basis2}
		a^{J,L}=\left[\frac{\p^{|I|}}{\p x^I}\left(\frac{1}{J!}(y-p)^J \frac{\p f^L}{\p e^K}\right)\right]_p.
	\end{equation}
	We conclude from \eqref{eq:basis2} that the matrix coefficients \( a^{J,L} \) of \( g_{\b\a} \) vary smoothly with \( p \).
\end{proof}

\begin{defn}\label{def:bundhom}
	The maps \( \Phi_p^{r,k} \), \( \D_p \) and \( \e_p \) combine together to form smooth vector bundle homomorphisms

	\[ \Phi^{r,k} = \Phi^{r,k}(E,\hat{\n},\n) : \otimes^{\leq r}(TM) \boxtimes \wedge^k(E) \to \U_k^r(E), \]

	\[ \D = \D(E) : \U_\bullet(E) \to \U_\bullet(E) \otimes_\R \U_\bullet(E), \] and \[ \e = \e(E): \U_\bullet(E)\to \U_0^0(E)\simeq M\times\R, \]
	
 	where \( \Phi^{r,k}\lfloor_{\otimes^{\leq r} (T_p M) \boxtimes \wedge^k(E_p)} = \Phi_p^{r,k} \), \( \D\lfloor_{\U_\bullet(E;p)}=\D_p \), and \( \e\lfloor_{\U_\bullet(E;p)}=\e_p \).
\end{defn}

That \( \Delta \) and \( \e \) are smooth bundle maps will follow once we know \( \Phi^{r,k} \) is a smooth bundle map, so we will show this first. In fact,

\begin{defn}
	If \( \widetilde{\hat{\n}}=({\hat{\n}}^p)_{p\in M} \) is a family of torsion-free connections on \( M \) and \( \widetilde{\n}=(\n^p)_{p\in M} \) is a family of connections on \( E \), let \[ \Phi^{r,k}(E,\widetilde{\hat{\n}},\widetilde{\n}): \otimes^{\leq r}(TM) \boxtimes \wedge^k(E) \to \U_k^r(E) \] be the map defined by \( \Phi^{r,k}(E,\widetilde{\hat{\n}},\widetilde{\n})\lfloor_{\otimes^{\leq r}(T_p M) \boxtimes \wedge^k(E_p)} = \Phi_p^{r,k}(E,\n^p,\widetilde{\n}^p) \).

\end{defn}

\begin{lem}\label{lem:varying}
	For such families \( \widetilde{\hat{\n}} \) and \( \widetilde{\n} \), the function \( \Phi^{r,k}(E,\widetilde{\hat{\n}},\widetilde{\n}) \) is a smooth bundle map.
\end{lem}

\begin{proof}
	Let \( \psi: \cal{N} \to \R^n \) be a coordinate chart on \( M \), let \( \eta: \cal{N} \times \R^d \to \pi^{-1}(\cal{N}) \) be a trivialization of \( E \) over \( \cal{N} \), and let \( \xi:\R^{n,d} \to \pi^{-1}(\cal{N}) \) be the bundle map \( \xi=\eta \circ (\psi^{-1} \times Id) \). We have, using the notation in the proof of Proposition \ref{prop:vb}, a trivialization of \( \U_k^r(E) \) over \( \cal{N} \), given at \( p\in \cal{N} \) by \[ \U_k^r(E;p) \xlongrightarrow{\xi_*^{-1}} \U_k^r(\R^{n,d}, \psi(p)) \xlongrightarrow{[\Phi_{\psi(p)}^{r,k}(\R^{n,d},\hat{\n}^{\R^n},\n^{\R^{n,d}})]^{-1}} S^{\leq r}(T_{\psi(p)}\R^n) \boxtimes \wedge^k(\R^{n,d}_{\psi(p)}) \xlongrightarrow{\simeq} S^{\leq r}(\R^n) \boxtimes \wedge^k(\R^d). \] Likewise, the bundle \( \otimes^{\leq r}(TM) \boxtimes \wedge^k(E)  \) is trivialized over \( \cal{N} \) by \[ \otimes^{\leq r}(T_pM) \boxtimes \wedge^k(E_p) \xlongrightarrow{\otimes^{\leq r}(T\psi)\boxtimes\wedge^k(\xi^{-1})} \otimes^{\leq r}(T_{\psi(p)}\R^n) \boxtimes \wedge^k (\R^{n,d}_{\psi(p)}) \xlongrightarrow{\simeq} \otimes^{\leq r}(\R^n) \boxtimes \wedge^k(\R^d). \] Let \[ F: \otimes^{\leq r}(\R^n) \boxtimes \wedge^k(\R^d) \to S^{\leq r}(\R^n) \boxtimes \wedge^k(\R^d) \] denote the composition of \( \Phi^{r,k}(E,\widetilde{\hat{\n}},\widetilde{\n}) \) with these trivializations. Fixing a basis element \( x_I\boxtimes e_K \) of \( \otimes^{\leq r}(\R^n) \boxtimes \wedge^k(\R^d) \), we have coefficients \( a^{J, L} \) such that \[ F(x_I \boxtimes e_K) = a^{J,L} x_J\boxtimes e_L\in S^{\leq r}(\R^n) \boxtimes \wedge^k(\R^d). \]
	
	By the defintion of \( \Phi^{r,k}(E,\widetilde{\hat{\n}},\widetilde{\n}) \), we have\footnote{The orders of the higher covariant derivatives that follow are suppressed.} \[ (\n^p_{x_I}\o)_p(e_K) = a^{J,L}(\n_{x_J}^\xi \o)_p(e_L) \] for any \( \o\in C^\i(\wedge^k E^*) \), where \( \n^\xi \) is the connection on \( \pi^{-1}(\cal{N}) \) induced by \( \xi \) from the Euclidean connection on \( \R^{n,d} \). In particular, fixing an \( n \)-dimensional multi-index \( J \) and a \( d \)-dimensional anti-symmetric index \( L \) and letting \( \o=\frac{1}{J!}(x-p)^J \cdot de^L \), by Lemma \ref{lem:euclidean},
	\begin{equation}\label{eq:smoothly}
		(\n^p_{x_I}\frac{1}{J!}(x-p)^J \cdot de^L)_p(e_K) = a^{J,L}.
	\end{equation}
	Since \( \n^p \) varies smoothly with \( p \), by Proposition \ref{prop:mega} so too do the higher covariant derivatives, and since \( \frac{1}{J!}(x-p)^J \cdot de^L \) also varies smoothy with \( p \), it follows from \eqref{eq:smoothly} that \( a^{J,L} \) varies smoothly with \( p \), hence \( F \) is smooth.
\end{proof}

By setting \( \hat{\n}^p=\hat{\n} \) and \( \n^p=\n \) for all \( p\in M \), we conclude
\begin{cor}
	The map \( \Phi^{r,k}(E,\hat{\n},\n) \) is a smooth bundle map of \( r \)-filtered \( k \)-graded bundles.
\end{cor}

The kernel \( \cal{K}^{r,k} = \cal{K}^{r,k}(E,\hat{\n},\n) \) of \( \Phi^{r,k}(E,\hat{\n},\n) \) is a smooth sub-bundle of \( \otimes^{\leq r}(TM) \boxtimes \wedge^k(E) \) since \( \widetilde{\Phi} \) has constant (full) rank, so we may take a quotient by \( \cal{K}^{r,k} \), and
\begin{cor}\label{cor:quot}
	\[ \left[\otimes^{\leq r}(TM) \boxtimes \wedge^k(E)\right] / \cal{K}^{r,k}(E,\hat{\n},\n) \simeq \U_k^r(E). \]
\end{cor}

\begin{proof}
	By Theorem \ref{thm:algebra}, the kernel of \( \Phi^{r,k}(E,\widetilde{\hat{\n}},\widetilde{\n}) \) is identical to the kernel of the natural projection \( \otimes^{\leq r}(TM) \boxtimes \wedge^k(E) \to S^{\leq r}(TM) \boxtimes \wedge^k(E) \). Since \( \Phi^{r,k}(E,\widetilde{\hat{\n}},\widetilde{\n}) \) is surjective, the result follows.
\end{proof}

\begin{defn}\label{def:coprodbox}
	The co-algebra structures on \( \otimes^\star(T_p M)\boxtimes \wedge^\bullet(E_p) \) for \( p\in M \) stitch together to form smooth bundle maps
	\begin{align*}
		\D_\otimes : \otimes^\star(TM)\boxtimes \wedge^\bullet(E) \to \left(\otimes^\star (TM)\boxtimes \wedge^\bullet(E)\right) \otimes \left(\otimes^\star(TM)\boxtimes \wedge^\bullet(E)\right)
	\end{align*}
	and 
	\begin{align*}
		\e_\otimes: \otimes^\star(TM)\boxtimes \wedge^\bullet(E) \to \otimes^0(TM)\boxtimes \wedge^0(E)\simeq M\times \R.
	\end{align*}
\end{defn}

\begin{cor}
	The functions \( \D:\U_\bullet(E) \to \U_\bullet(E) \otimes \U_\bullet(E) \) and \( \e: \U_\bullet(E)\to \U_0^0(E)\simeq M\times\R \) given by Definition \ref{def:bundhom} are smooth bundle maps. Moreover, \[ \D\circ \Phi(E,\hat{\n},\n)= \Phi(E,\hat{\n},\n)^{\otimes 2}\circ \D_\otimes,  \] and \[ \e\circ \Phi(E,\hat{\n},\n)=\Phi(E,\hat{\n},\n)\circ \e_\otimes. \]
\end{cor}

\begin{proof}
	By Theorem \ref{thm:co-alg}, \( \D_\otimes \) (resp. \( \e_\otimes \)) descends via \( \Phi \) to \( \D \) (resp. \( \e \),) and since \( \Phi \) is a smooth surjective submersion, it follows that \( \D \) and \( \e \) are also smooth. The two equations both follow immediately from Theorem \ref{thm:co-alg}.
\end{proof}

Now suppose \( \widetilde{\n}=(\hat{\n}^p)_{p\in M} \) is a family of connections on \( E \) and \( \widetilde{\hat{\n}}=(\hat{\n}^p)_{p\in M} \)) is a family of torsion-free connections on \( TM \), such that \( E \) and \( M \) are both asymptotically flat in a neighborhood of \( p \) for each \( p\in M \) (such families trivially exist as \( \hat{\n}^p \) and \( \n^p \) can defined to be flat on a neighborhood \( p \) and extended to the rest of \( M \) using a partition of unity.) 
\begin{cor}\label{cor:unnatural}
	The bundle map \( \Phi^{r,k}(E,\widetilde{\hat{\n}},\widetilde{\n}) \) descends to an isomorphism of smooth vector bundles \[ [\Phi^{r,k}(E,\widetilde{\hat{\n}},\widetilde{\n})]: S^{\leq r}(TM) \boxtimes \wedge^k(E) \to \U_k^r(E).  \]
\end{cor}

\subsection{Orientation of \( \U \)}


If \( E \) and \( M \) are oriented, then \( \cal{K}^{r,k}(E,\hat{\n},\n) \) is oriented, since the change-of-basis matrix for \( \cal{K}^{r,k} \) under a change of coordinates is, by \eqref{eq:changeofbasis}, block-upper triangular, whose diagonal terms are the same as those for the kernel of the natural projection \( \otimes^r (TM) \boxtimes \wedge^k (E) \to S^r (TM) \boxtimes \wedge^k (E) \). It follows that \( \U_k^r(E,\hat{\n},\n) \) is also oriented.

\subsection{Morphisms of \( \U \)}
Suppose \( F\to N \) is a smooth vector bundle and \( \Theta: E\to F \) is a homomorphism covering \( \theta: M\to N \). The \emph\textbf{{pushforward}} \( \Theta_* : \U_k^r(E;p) \to \U_k^r(F;\theta(p)) \) is defined: For \( T\in \U_k^r(E;p) \) and \( \o\in C^\i(\wedge^k(F^*)) \), let \( \Theta_* T (\o):= T(\wedge^k(\Theta^*)\circ \o) \). This defines a bundle map \( \Theta_*: \U_k^r(E)\to \U_k^r(F) \) and turns \( \U \) into a functor. Moreover, if \( \o\in C^\i(\wedge^k(F^*)) \) and \( \eta\in C^\i(\wedge^j(F^*)) \), then \( \wedge^{k+j}(\Theta^*)\circ (\o\wedge\eta) = (\wedge^{k}(\Theta^*)\circ (\o))\wedge (\wedge^{j}(\Theta^*)\circ (\eta)) \) (i.e. pullback commutes with wedge product.) This implies,  

\begin{prop}\label{prop:pushforward}
	The bundle maps \( \Theta_* \) and \( \Delta \) commute.
\end{prop}

In the case \( N=M \) we will consider more general bundle homomorphisms on \( \U(E) \) arising from differential operators:

\begin{prop}\label{prop:diffop}
	\begin{enumerate}
		\item A bundle homomorphism \( \cal{H} : \U(E)\to \U(F) \) covering the identity determines a differential operator \( \cal{H}^*: C^\i_c(\wedge(F^*)) \to C^\i_c(\wedge(E^*)) \) by defining \( (\cal{H}^*\o)_p(\a)=\o_p(\cal{H}(\a)) \) for \( \a\in \wedge(E_p)\simeq \U^0(E;p) \).
		\item If \( T: C^\i_c(\wedge(F^*)) \to C^\i_c(\wedge(E^*)) \) is a differential operator, then the algebraic dual of \( T \) defines a bundle homomorphism \( T^*: \U(E)\to \U(F) \) covering the identity. Moreover, \( (T^*)^*=T \). If \( T \) is a differential operator of order \( s \), then \( T^* \) restricts to a bundle homomorphism \( \U^r(E)\to \U^{r+s}(F) \).
	\end{enumerate}
\end{prop}

\begin{defn}\label{def:consistent}
	It is not true that \( (\cal{H}^*)^*=\cal{H} \) for a general bundle homomorphism \( \cal{H} : \U(E)\to \U(F) \) covering the identity, since the left hand side is determined entirely by the value of \( \cal{H} \) on \( \U^0(E) \). We will call homomorphisms that satisfy this equality \emph{consistent}. 
\end{defn}

By Proposition \ref{prop:diffop}, a bundle homomorphism \( \cal{H}: \U(E)\to \U(F) \) covering the identity is consistent if and only if it is the dual of a differential operator.

\subsection{Module Structures}
The space \( C^\i(M) \) acts on each fiber \( \U_k^r(E;p) \) of \( \U_k^r(E) \) by duality. That is, if \( f\in C^\i(M) \) and \( T\in \U_k^r(E;p) \), then for \( \o\in C^\i(\wedge^k(E^*)) \), let \( f\lrcorner T (\o) := T(f\o) \). This turns \( \U_k^r(E;p) \) into a \( C^\i(M) \)-module, and each \( f\in C^\i(M) \) defines a bundle endomorphism \( f\lrcorner: \U_k^r(E)\to \U_k^r(E) \). The bundle endomorphism \( f\lrcorner \) lifts via \( \Phi \) to a bundle endomorphism, which we will also denote by \( f\lrcorner \), of \( \otimes(TM)\boxtimes \wedge(E) \), given by \( f\lrcorner (v\otimes \a)=(\n_{v_{(1)}} f) v_{(2)}\boxtimes \a \).

As a corollary to Proposition \ref{prop:pushforward}, we have
\begin{cor}\label{cor:pushforward}
	Pushforward commutes with \( f\lrcorner \) on \( \U(E) \). That is, if \( f\in C^\i(N) \) and \( \Theta: E\to F \) is a smooth vector bundle homomorphism covering \( \theta: M\to N \), then \( \Theta_* [(f\circ\theta)\lrcorner T] = f\lrcorner (\Theta_* T) \) for all \( T\in \U(E) \).
\end{cor}

The space \( C^\i(\U_k^r(E)) \) has a \( C^\i(M) \)-module structure given by \( f\lrcorner \). Note that is not the standard \( C^\i(M) \)-module structure on the space of sections of a vector bundle over \( M \). Pushforward by a smooth vector bundle homomorphism covering the identity is a module homomorphism by Corollary \ref{cor:pushforward}. 

\begin{thm}
	With the module structure given by \( f\lrcorner \), the space \( C^\i(\U_k^r(E)) \) is a finitely generated projective \( C^\i(M) \)-module.
\end{thm}

\begin{proof}
	Let \( E'\to M \) be a smooth vector bundle such that the Whitney sum \( H=E\oplus E' \) is trivial (such a vector bundle always exists even if \( M \) is non-compact.) Let \( \iota: E\to H \) and \( \pi: H\to E \) be the canonical injection and projection, respectively. Using the Whitney embedding theorem, assume \( M \) is embedded in \( \R^N \) for some \( N>0 \) and let \( U\subset \R^N \) be a tubular neighborhood of \( M \) (with possibly non-constant radius, since \( M \) is not assumed to be compact.) Let \( \xi: M\to U \) be the inclusion of \( M \) into \( U \), and let \( \upsilon: U\to M \) be the retraction map defined by the tubular neighborhood theorem.
	
	Observe that the pullback bundle \( \check{H}:=\upsilon^* H \) over \( U \) is also trivial. Let \( \Xi: H\to \check{H} \) be the vector bundle homomorphism \( (\xi, Id) \), and let \( \Upsilon: \check{H}\to H \) be the vector bundle homomorphism \( (\upsilon, Id) \). 
	
	Let \( R=\xi^* \U_k^r(\check{H}) \), the restriction of \( \U_k^r(\check{H}) \) to \( M \). Since \( f\lrcorner \) acts on fibers, the action restricts to \( R \). The bundle map \( \Xi_*: \U_k^r(H)\to \U_k^r(\check{H})\) factors through \( R \), so let \( X: \U_k^r(H) \to R  \) denote this bundle map with restricted codomain.
	
	Likewise, let \( T: R\to \U_k^r(H) \) denote the restriction of \( \Upsilon_* \) to \( R \) (i.e. the pullback by \( \xi \) of \( \Upsilon_* \).) Since \( \Upsilon \circ \Xi \) is the identity on \( H \), it follows that \( T\circ X \) is the identity on \( \U_k^r(H) \).
	
	Moreover, Corollary \ref{cor:pushforward} implies \( X \) and \( T \) commute with \( f\lrcorner \). Since \( \iota_* \), \( \pi_* \), \( X \) and \( T \) are all smooth bundle maps covering the identity on \( M \) which commute with \( f\lrcorner \), we thus have \( C^\i(M) \)-linear module homomorphisms
	\begin{equation}
		 \begin{tikzcd}[column sep=15ex]
			C^\i(\U_k^r(E)) \arrow{r}{C^\i(\iota_*)} &  C^\i(\U_k^r(H)) \arrow{r}{C^\i(X)} & C^\i(R)
		\end{tikzcd}
	\end{equation}
	and
	\begin{equation}
		 \begin{tikzcd}[column sep=15ex]
			C^\i(R) \arrow{r}{C^\i(T)} & C^\i(\U_k^r(H)) \arrow{r}{C^\i(\pi_*)} & C^\i(\U_k^r(E)),
		\end{tikzcd}
	\end{equation}
	which by functoriality compose to the identity on \( C^\i(\U_k^r(E)) \). The top map in particular is an injection and the bottom map is surjective. Thus, by the splitting lemma, \( C^\i(\U_k^r(E)) \) is a direct summand of \( C^\i(R) \), so it suffices to show that \( C^\i(R) \) is free and finitely generated.
	
	The bundle \( R \) is isomorphic via Corollary \ref{cor:unnatural} to the trivial bundle \( M\times S^{\leq r} \R^N \otimes \wedge^k \R^K \), where \( K \) is the dimension of \( H \). By Proposition \ref{prop:leibniz}, the induced action of \( f\lrcorner \) on \( M\times S^{\leq r} \R^N \otimes \wedge^k \R^K \) is given by
	\[ f\lrcorner(p, x\otimes \alpha) = (p, [\p_{x_{(1)}}f](p)\cdot x_{(2)}\otimes \alpha), \]
	where \( x\in S^{\leq r}\R^N \) and \( \a\in \wedge^k \R^K \). So, we are reduced to showing that \( C^\i (M\times S^{\leq r} \R^N \otimes \wedge^k \R^K) \) is free and finitely generated over \( C^\i(M) \).

	Claim: The set \( S \) of constant sections \( \{ x_s\otimes \a_t \} \) of \( M\times S^{\leq r} \R^N \otimes \wedge^k \R^K \), where \( x_s \) is the standard basis element (over \( \R \)) of \( S^{\leq r}\R^N \), and \( \a_t  \) is a standard basis element (again, over \( \R \)) of \( \wedge^k \R^K \), is a basis of \( C^\i (M\times S^{\leq r} \R^N \otimes \wedge^k \R^K) \) (over \( C^\i(M) \).)
	
	We first show \( S \) is linearly independent. Suppose
	\begin{equation}\label{eq:li}
		 f^{s,t}\lrcorner(x_s\otimes \a_t)=0.
	\end{equation}
	Let \( 0\leq r'\leq r \) and suppose we have shown \( f^{s,t}=0 \) for all \( |s|>r' \) (The case that \( r'=r \) is vacuous.) Let us show that \( f^{s,t}=0 \) for all \( |s|=r' \). Expanding \eqref{eq:li}, the value of \( f^{s,t}\lrcorner(x_s\otimes \a_t) \) at \( p\in M \) is 
	\[
	\sum_{|s|=r'} f^{s,t}(p)\cdot x_s \otimes \a_t + w,
	\]
	where \( w \) consists of terms with symmetric power strictly less than \( r' \). Since \( \{ x_s\otimes \a_t \} \) is a basis of the fiber above \( p \), it follows that \( f^{s,t}(p)=0 \) for all \( |s|=r' \), and since this is true for each \( p\in M \), \( f^{s,t}=0 \) on all \( M \). We conclude by downward induction that \( f^{s,t}=0 \) for all \( s \) and \( t \), thus \( S \) is linearly independent.
	
	Let \( |S| \) denote the submodule of \( C^\i (M\times S^{\leq r} \R^N \otimes \wedge^k \R^K) \) generated by \( S \). To see that \( |S|= C^\i (M\times S^{\leq r} \R^N \otimes \wedge^k \R^K) \), we proceed again by induction but this time from inducting upward.
	
	First, since the value of \( f\lrcorner(1\otimes \a_t) \) at \( p \) is \( f(p)\otimes \a_t \), we conclude that \( C^\i (M\times S^0 \R^N \otimes \wedge^k \R^K) \) is contained in \( |S| \). If \( r=0 \) we are done, otherwise let \( 0< r'\leq r \) and suppose \( C^\i (M\times S^{< r'} \R^N \otimes \wedge^k \R^K) \) is contained in \( |S| \). An arbitrary section of \( M\times S^{r'}\R^N \otimes \wedge^k \R^K) \) can be written as \( f^{s,t} x_s\otimes \a_t \) where \( |s|=r' \) and \( f^{s,t}\in C^\i(M) \) for all \( s \) and \( t \). On the other hand, \[ f^{s,t}\lrcorner(x_s\otimes \a_t) = f^{s,t} x_s\otimes \a_t + [\p_{{x_s}_{(1')}}f] {x_s}_{(2')}\otimes \a_t, \]
where \( {x_s}_{(1')} \) and \( {x_s}_{(2')} \) are the Sweedler notation factors of \( \Delta x_s - 1\otimes x_s \). The second term on the right hand side is contained in \( |S| \) by our inductive hypothesis, and so is the term on the left, leaving \( f^{s,t} x_s\otimes \a_t \), which therefore must be as well.
\end{proof}


\begin{defn}\label{def:LfRf}
	Let \( F\to M \) be a smooth vector bundle and consider the actions \( L_f \) and \( R_f \) of \( C^\i(M) \) on \( Hom(\U_k^r(E),\U_j^s(F)) \) given by pre- and post-composition by \( f\lrcorner \).
\end{defn}

A similar proof shows:
\begin{thm}
	The space \( C^\i(Hom(\U_k^r(E),\U_j^s(F))) \) is a finitely generated projective module, using either \( L_f \) or \( R_f \) for the module structure. The consistent homomorphisms (see Definition \ref{def:consistent}) are a sub-bundle, again using either \( L_f \) or \( R_f \) for the module structure.
\end{thm}

\section{Distinguished Endomorphisms of \( \U \) - Interior Product and Higher Covariant Differentiation}\label{sec:distinguished}
We will now derive formulas for actions on \( \U(E) \) arising from interior product and higher covariant differentiation. These will be defined as actions on \( \otimes(TM) \boxtimes \wedge(E) \), and they will pass through the projection \( \Phi \) to actions on \( \U(E) \) by consistent endomorphisms. In the case \( E=TM \) and in the presence of a semi-Riemannian metric on \( M \), we will also compute adjoints for these actions.

\subsection{Interior Product}
Let \( X\in C^\i(\wedge(E)) \) and let \( \mathbb{E}_X \) be the vector bundle endomorphism of \( \otimes(TM) \boxtimes \wedge(E) \) defined for \( v\boxtimes \a\in \otimes(T_p M) \boxtimes \wedge(E_p) \) by
\begin{equation*}
	\mathbb{E}_X(v\boxtimes \a):=v_{(1)}\boxtimes (\n_{\tilde{v}_{(2)}} X)\wedge\a,
\end{equation*}
where as usual we are using Sweedler notation for the tensor co-product on \( \otimes(T_pM) \), and \( \tilde{v} \) is any smooth extension of \( v \). By Corollary \ref{cor:interior},
\begin{prop}\label{prop:bosoncreation}
	Let \( \o\in C^\i(\wedge(E^*)) \). If \( p\in M \) and \( v\boxtimes \a \in \otimes(T_pM) \boxtimes \wedge(E_p) \), then \[ \Phi_p(\mathbb{E}_X(v\boxtimes \a))\o=\Phi_p (v\boxtimes \a)(\iota_X \o). \]
\end{prop}

In particular, \( \mathbb{E}_X \) preserves the kernel \( \cal{K} \) of \( \Phi \), and thus passes to an endomorphism of \( \U(E) \), which we will denote by \( [\mathbb{E}_X] \). By Proposition \ref{prop:bosoncreation},

\begin{cor}\label{cor:exdual}
	\( [\mathbb{E}_X] \) is dual to interior product \( \iota_X \). More precisely, the bundle map \( [\mathbb{E}_X] \) restricted to a fiber \( \U(E;p) \) is dual, using the natural pairing \( (\U(E;p), C^\i(\wedge(E^*))) \), to \( \iota_X \). 
\end{cor}

Calculating directly from the definition of \( \mathbb{E} \), we have for \( X, X'\in C^\i(\wedge(E)) \),
\begin{equation}
	\mathbb{E}_X\circ \mathbb{E}_{X'} = \mathbb{E}_{X\wedge X'},\label{eq:comme},
\end{equation}
and treating a function \( f\in C^\i(M) \) as a section of \( \wedge^0(E) \), we have

\begin{equation}\label{eq:comme1}
	\mathbb{E}_f =f\lrcorner.
\end{equation}

If we let \( \mathbb{E}: C^\i(\wedge(E)) \to C^\i(End(\otimes(TM)\boxtimes \wedge(E))) \) be the map \( X\mapsto \mathbb{E}_X \), we conclude from \eqref{eq:comme} that \( \mathbb{E} \) is a (left) \( C^\i(M) \)-algebra action of \( C^\i(\wedge(E)) \) on \( \otimes(TM)\boxtimes \wedge(E) \), where the domain is equipped with wedge product and the standard scalar-multiplication by smooth functions, and the codomain is equipped with composition and scalar-multiplication by either \( L_f \) or \( R_f \) (See Definition \ref{def:LfRf}.) Moreover, \( \mathbb{E} \) is local, i.e. \( supp(\mathbb{E}_X)\subset \supp(X) \).

\subsection{Covariant Differentiation}
Let \( Y\in C^\i(\otimes(TM)) \) and let \( \mathbb{D}_Y \) be the vector bundle endomorphism of \( \otimes(TM) \boxtimes \wedge(E) \), defined for \( v\boxtimes \a\in \otimes(T_p M) \boxtimes \wedge(E_p) \) by
\begin{equation*}
	\mathbb{D}_Y(v\boxtimes \a):=v_{(1)}\otimes\left(\n_{\tilde{v}_{(2)}}Y\right) \boxtimes \a.
\end{equation*}
By Proposition \ref{prop:mega},

\begin{prop}\label{prop:fermioncreation}
	If \( \o\in C^\i(\wedge(E^*)) \), \( p\in M \) and \( v\boxtimes \a \in \otimes(T_pM) \boxtimes \wedge(E_p) \), then \[ \Phi_p(\mathbb{D}_Y(v\boxtimes \a))\o=\Phi_p (v\boxtimes \a)(\n_Y \o). \]
\end{prop}

Therefore, \( \mathbb{D}_Y \) preserves the sub-bundle \( \cal{K} \) and passes to an endomorphism of \( \U(E) \), which we will denote by \( [\mathbb{D}_Y] \). We conclude from Proposition \ref{prop:fermioncreation} that
\begin{cor}\label{cor:dydual}
	\( [\mathbb{D}_Y] \) is dual to higher covariant derivative \( \n_Y \) on \( C^\i(\wedge(E^*)) \).
\end{cor}

As a special case, treating a function \( f\in C^\i(M) \) as a section of \( \otimes^0(TM) \), we have \( \mathbb{D}_f=f\lrcorner \).

Calculating directly from the definitions of \( \mathbb{E} \) and \( \mathbb{D} \), we have for \( X\in C^\i(\wedge(E)) \) and \( Y, Y'\in C^\i(\otimes(TM)) \),
\begin{equation}
	\mathbb{D}_Y\circ \mathbb{D}_{Y'} = \mathbb{D}_{Y'_{(1)}\n_{Y'_{(2)}}Y},\label{eq:commd}
\end{equation}
and 
\begin{equation}
	\mathbb{E}_X\circ \mathbb{D}_Y = \mathbb{D}_{Y_{(1)}}\circ \mathbb{E}_{\n_{Y_{(2)}} X}.\label{eq:comm1}
\end{equation}

In particular, if \( Y\in C^\i(E) \), then \[ \mathbb{E}_X \circ \mathbb{D}_Y - \mathbb{D}_Y \circ \mathbb{E}_X  =\mathbb{E}_{\n_Y X}. \]

Letting \( \mathbb{D}: C^\i(\otimes(TM))\to C^\i(End(\otimes(TM)\boxtimes \wedge(E))) \) be the map \( Y\mapsto \mathbb{D}_Y \), we conclude from \eqref{eq:commd} that \( \mathbb{D} \) is a right \( \R \)-algebra action of \( C^\i(\otimes(TM)) \) on \( \otimes(TM)\boxtimes \wedge(E) \), where the domain is equipped with covariant product. Moreover, \( \mathbb{D} \) is local and \( C^\i(M) \)-linear where the domain has the standard multiplication by smooth functions, and the codomain has scalar multiplication given by \( R_f \).

\subsection{Adjoint Endomorphisms}
We will suppose throughout this sub-section that \( M \) is an oriented semi-Riemannian manifold without boundary, that \( \n \) is the Levi-Civita connection on \( E=TM \) and \( \Phi=\Phi(TM, \n,\n) \). Our goal will be to define endomorphisms on \( \U(M) \) adjoint to interior product and covariant differentiation. To define these, we will need to define a notion of adjoint in this context.

\begin{defn}\label{def:adjoint}
	Suppose \( F \) is a consistent endomorphism of \( \U(M) \) and let \( (F^*)^\dagger: \cal{D}^\bullet(M) \to \cal{D}^\bullet(M) \) be the \( L^2 \) adjoint of the differential operator \( F^* : \cal{D}^\bullet(M) \to \cal{D}^\bullet(M) \). Let \( F^\dagger = ((F^*)^\dagger)^* \). We call \( F^\dagger \) the \emph{adjoint} of \( F \).
\end{defn}

By Proposition \ref{prop:diffop},
\begin{prop}\label{prop:adjointidentities}
	If \( F, G \in C^\i(End(\U(M))) \) are consistent, then \( F\circ G \) and \( F^\dagger \) are consistent. Moreover, \( (F\circ G)^\dagger = G^\dagger \circ F^\dagger \) and \( (F^\dagger)^\dagger=F \).
\end{prop}

By Corollaries \ref{cor:exdual} and \ref{cor:dydual}, the endomorphisms \( [\mathbb{E}_X] \) and \( [\mathbb{D}_Y] \) are both consistent and thus have adjoints. To derive explicit formulas for them, we will need to dualize the Hodge star operator on \( \cal{D}^\bullet \) to yield an endomorphism of \( \U_\bullet(M) \).

\begin{lem}\label{lem:hodgestar}
	Suppose \( V \) is an oriented real vector space of dimension \( n<\i \), equipped with a non-degenerate bilinear pairing \( <\cdot,\cdot> \). Let \( \hat{\star}: \wedge(V) \to \wedge(V) \) be the Hodge Star map, and let \( \star: \wedge(V^*) \to \wedge(V^*) \) be the Hodge Star on the dual space \( V^* \) (equipped with the bilinear pairing arising from the isomorphism \( V\simeq V^* \) induced by \( <\cdot,\cdot> \).) If \( \wedge(V^*) \) is identified with \( (\wedge(V))^* \) by the pairing \( (\a^1\wedge\cdots\wedge \a^k, v_1,\dots,v_k) = det(\a^i(v_j)) \), then \[ \hat{\star}^* = \star^{-1}. \]
\end{lem}

\begin{proof}
	Let \( (x_1,\dots, x_n) \) be an ordered orthonormal basis of \( V \) and let \( (x^1,\dots, x^n) \) denote the dual basis of \( V^* \). Let \( I=(i_1,\dots, i_k) \) and \( J=(j_1,\dots, j_{n-k}) \) be ordered subsets of \( (1,\dots, n) \) and let \( x^I=x^{i_1}\wedge\cdots\wedge x^{i_k} \) (resp. \( X_J=x_{j_1}\wedge\cdots\wedge x_{j_{n-k}} \).) Then, \( x^I(\hat{\star} x_J) \) and \( \star^{-1}(x^I)(x_J) \) are both zero unless \( J=(1,\dots, n)\setminus I \), in which case \[ x^I(\hat{\star} x_J) = x^I(sgn(\sigma)(-1)^{k(n-k)}s(J) x_I) = sgn(\sigma)(-1)^{k(n-k)}s(J), \] where \( \sigma \) is the permutation \( IJ \) and \( s(J) \) is the product \( \prod_{j\in J}<x_j,x_j> \). On the other hand, \[ \star(x^I) (x_J) = sgn(\sigma)s(I), \] and so \( \hat{\star}^* = (-1)^{k(n-k)}s(1,\dots,n)\star= \star^{-1}. \)
\end{proof}

By Proposition \ref{prop:diffop}, Hodge star \( \star: \cal{D}^\bullet(M)\to \cal{D}^{n-\bullet}(M) \) dualizes to an endomorphism of \( \U(M) \) which we will denote by \( [\perp] \). Since \( \star \) commutes with the Levi-Civita covariant derivative, we know by Lemma \ref{lem:hodgestar} that \( [\perp] \) lifts via \( \Phi \) to a bundle endomorphism \( \perp \) of \( \otimes (TM) \boxtimes \wedge (TM) \), given by the following formula:
\begin{equation}\label{eq:perp}
	\perp(v\boxtimes \a) = v\boxtimes \star^{-1}(\a),
\end{equation}
for \( v\in \otimes (T_pM) \) and \( \a\in \wedge(T_pM) \), where in \eqref{eq:perp}, the map \( \star \) is the usual Hodge star on \( \wedge(T_p M) \).

\subsubsection{Adjoint of Interior Product}
Using the \( L^2 \) inner product on differential forms, we know that the adjoint of interior product \( \iota_X \) for a vector field \( X\in C^\i(TM) \) is given on \( k \)-forms by \( \iota_X^\dagger = (-1)^k \star^{-1}\iota_X \star \). More generally, by iterating using \( \iota_{\alpha\wedge \beta}=\iota_\beta\circ\iota_\alpha \), we have \[ \iota_X^\dagger = (-1)^{rk} \star^{-1}\iota_X\star, \] where \( X \in C^\i(\wedge^r(TM)) \), again acting on \( k \)-forms. 

Therefore, the adjoint of \( [\mathbb{E}_X] \) is given on \( \U_k(M) \) by \[ [\mathbb{E}_X]^\dagger = (-1)^{r(k+r)}[\perp] \circ[\mathbb{E}_X] \circ[\perp]^{-1}, \] which lifts via \( \Phi \) to the endomorphism \[ \mathbb{E}_X^\dagger := (-1)^{r(k+r)}\perp\circ \mathbb{E}_X \circ\perp^{-1} \] of \( \otimes(TM)\boxtimes \wedge(TM) \).

A calculation in coordinates shows that if \( \a_1,\dots,\a_k\in T_pM \), then \[ \mathbb{E}_X^\dagger(\a_1\wedge\dots\wedge \a_k) = \sum_{L_r\subset (1,\dots,k)} (-1)^{l_1+\cdots+l_r+r(r+1)/2}<X(p),\a_{L_r}> \cdot \, \a_1\wedge\cdots\wedge\widehat{\a_{L_r}}\wedge\cdots\wedge \a_k, \] where the sum is taken over all length-\( r \) ordered subsets \( L_r=(l_1,\dots,l_r) \) of \( (1,\dots,k) \), and \( \a_{L_r}=\a_{l_1}\wedge\dots\wedge \a_{l_r} \). The notation \( \a_1\wedge\cdots\wedge\widehat{\a_{L_r}}\wedge\cdots\wedge \a_k \) refers to the \( (k-r) \)-vector resulting from wedging together the vectors \( \a_i \), \( i\in (1\dots,k)\setminus L_r \).

If \( r=1 \), this simplifies to \[ \mathbb{E}_X^\dagger(\a_1\wedge\cdots\wedge \a_k)=\sum_{l=1}^k<X(p),\a_l>(-1)^{l+1}\cdot \a_1\wedge \cdots\wedge \widehat{\a_l}\wedge\cdots\wedge \a_k. \]

By \eqref{eq:comm1},
\begin{equation}\label{eq:comm2}
	\mathbb{E}_X^\dagger\circ \mathbb{D}_Y = \mathbb{D}_{Y_{(1)}}\circ \mathbb{E}_{\n_{Y_{(2)}} X}^\dagger,
\end{equation}

so for a general element \( v\boxtimes \a \) of \( \otimes(TM)\boxtimes \wedge(TM) \), 

\begin{equation}\label{exdagger}
	\mathbb{E}_X^\dagger(v\boxtimes \a) = v_{(1)}\boxtimes \mathbb{E}_{\n_{v_{(2)}}X}^\dagger \a.
\end{equation}

Additionally, from the definition of \( \mathbb{E}_X^\dagger \), we know
\begin{equation}\label{eq:compedagger}
	\mathbb{E}_X^\dagger \circ \mathbb{E}_{X'}^\dagger = \mathbb{E}_{X'\wedge X}^\dagger.
\end{equation}
A computation in coordinates shows:
\begin{prop}\label{prop:eedagger}
	If \( X, Y\in C^\i(TM) \), then \( \{\mathbb{E}_X, \mathbb{E}_Y^\dagger\}=<X,Y>\lrcorner \), where \( <X,Y> \) is the function on \( M \) whose value at \( p \) is \( g(X(p),Y(p)) \) where \( g \) is the Riemannian metric on \( M \).
\end{prop} 

Thus, \[ \mathbb{E}+\mathbb{E}^\dagger: C^\i(TM)\to C^\i(End(\otimes(TM)\boxtimes \wedge(TM))) \] \[ X\mapsto \mathbb{E}_X + \mathbb{E}_X^\dagger \] extends to a \( C^\i(M) \)-algebra map from smooth sections of the Clifford bundle \( Cl(TM) \) to \(  C^\i(End(\otimes(TM)\boxtimes \wedge(TM))) \), and the image (up to a sign) of the volume form under this map is \( \perp \). That is, locally on \( \otimes(TM)\boxtimes \wedge^k(TM) \),

\[ \perp = (-1)^{k(k-1)/2} \prod_{i=1}^n <e_i,e_i>(\mathbb{E}_{e_i} + \mathbb{E}_{e_i}^\dagger), \] where \( (e_i) \) is an orthonormal local frame.

\subsubsection{Adjoint of Covariant Differentiation}
For a vector field \( X\in C^\i(TM) \), the adjoint of covariant differentiation \( \n_X \) on differential forms is the map \( \eta \mapsto -div(X)\eta - \n_X \eta \). Therefore, \( [\mathbb{D}_X]^\dagger = -div(X)\lrcorner - [\mathbb{D}_X] \). So, let us define \( \mathbb{D}_X^\dagger\in C^\i(End(\otimes(TM)\boxtimes\wedge(TM))) \) to be the map
\[ \mathbb{D}_X^\dagger := -div(X)\lrcorner - \mathbb{D}_X. \]
This is given explicitly as
\begin{equation}\label{eq:dxdagger}
	\mathbb{D}_X^\dagger(v\otimes \a)= -v_{(1)}\n_{\tilde{v}_{(2)}}(div(X)+X)\otimes \a.
\end{equation}

To define \( \mathbb{D}_X^\dagger \) for more general tensor fields \( X\in C^\i(\otimes(TM)) \), we use the composition property \[ [\mathbb{D}_X]^\dagger \circ [\mathbb{D}_Y]^\dagger = ([\mathbb{D}_Y]\circ [\mathbb{D}_X])^\dagger = [\mathbb{D}_{X_{(1)}\n_{X_{(2)}} Y}]^\dagger \] recursively and define \[ \mathbb{D}_{XY}^\dagger:=\mathbb{D}_X^\dagger \circ \mathbb{D}_Y^\dagger-\mathbb{D}_{X_{(1')}\n_{X_{(2')}} Y}^\dagger, \] where \( X_{(1')} \) and \( X_{(2')} \) are the Sweedler factors of \( \Delta X - X\otimes 1 \). However, a closed-form expression like \eqref{eq:dxdagger} seems difficult to derive in this more general case.

For vector fields \( X \) and \( Y \) on \( M \), \[ [\mathbb{D}_X, \mathbb{D}_Y^\dagger] = X(div(Y))\lrcorner - \mathbb{R}_{X,Y}+ \mathbb{D}_{[X,Y]}, \] 

\[ [\mathbb{D}_X, \mathbb{D}_Y] = \mathbb{R}_{X,Y} - \mathbb{D}_{[X,Y]}, \] and \[ [\mathbb{D}_X^\dagger, \mathbb{D}_Y^\dagger] = - div [X,Y]\lrcorner + \mathbb{R}_{X,Y} - \mathbb{D}_{[X,Y]}, \]
where \( \mathbb{R}_{X,Y}=\mathbb{D}_{Y\otimes X - X\otimes Y} \).

\subsection{Combined Actions}\label{section:combined}
Let \( B=C^\i(\otimes(TM)\boxtimes \wedge(E)) \) and consider the map \( \sharp: B\otimes B\to B \) given by 
\begin{equation}\label{sharptimes}
	(v\boxtimes \a) \sharp (w\boxtimes \b) =  (w_{(1)}\odot v)\boxtimes (\n_{w_{(2)}}\a)\wedge \b.
\end{equation}
In \eqref{sharptimes}, the square tensor products are over \( C^\i(M) \). Note that this takes the form of the Hopf-algebraic smash product, except that the co-product on \( C^\i(\wedge(E)) \) is over \( C^\i(M) \) and not \( \R \).

\begin{prop}
	The map \( \sharp \) is well-defined and turns \( B \) into a filtered graded unital associative \( \R \)-algebra.
\end{prop}

Consider the map \( \mathbb{DE}: B \to C^\i(End(\otimes(TM)\boxtimes \wedge(E))) \) given by \( \mathbb{DE}_{v\boxtimes \a}= \mathbb{D}_v\circ \mathbb{E}_\a \).

\begin{prop}
	The map \( \mathbb{DE} \) is a well-defined left \( \R \)-algebra action of \( B \) on the bundle \( \otimes(TM)\boxtimes \wedge(E) \). Moreover, \( \mathbb{DE} \) descends via \( \Phi \) to a left \( \R \)-algebra action \( [\mathbb{DE}] \) on \( \U(E) \) given by \( [\mathbb{DE}]_{v\boxtimes \a}=[\mathbb{D}]_v\circ [\mathbb{E}]_\a \). Moreover, \( \mathbb{DE} \) and \( [\mathbb{DE}] \) are local and \( C^\i(M) \)-linear if the codomain has scalar multiplication given by \( R_f \).
\end{prop}

Specializing to the case that \( E=T^*M \), the trace of the restriction of \( \mathbb{DE} \) to \( C^\i(TM\boxtimes T^*M) \) is a well-defined endomorphism of \( \otimes(TM)\boxtimes \wedge(T^*M) \), which we will denote \( tr(\mathbb{DE}) \). The map \( tr(\mathbb{DE}) \) descends via \( \Phi \) a well-defined endomorphism of \( \U(T^*M) \), denoted by \( [tr(\mathbb{DE})] \). This endomorphism is the trace of the restriction of \( [\mathbb{DE}] \) to \( C^\i(TM\boxtimes T^*M) \). That is, \( tr([\mathbb{DE}])=[tr(\mathbb{DE})] \). Identifying \( \U(T^*M) \) with \( \U(TM)=\U(M) \) using the metric, \( tr([\mathbb{DE}]) \) is dual (in the sense of Proposition \ref{prop:diffop}) to \( -\delta \), where \( \delta=(-1)^k \star^{-1}d\star \) is the co-differential on \( k \)-forms.

Now consider the map \( \mathbb{DE^\dagger} : C^\i(\otimes(TM)\boxtimes \wedge(E)) \to C^\i(End(\otimes(TM)\boxtimes \wedge(E))) \) given by \( \mathbb{DE^\dagger}_{v\boxtimes \a}= \mathbb{D}_v\circ \mathbb{E^\dagger}_\a \).

\begin{prop}
	The map \( \mathbb{DE^\dagger} \) is a well-defined left \( \R \)-algebra action of \( C^\i(\otimes(TM)\boxtimes \wedge(E)) \) on \( \otimes(TM)\boxtimes \wedge(E) \), up to sign. Specifically, 
	\[ \mathbb{DE^\dagger}_{(v\boxtimes \a)\sharp (w\boxtimes \b)}=(-1)^{|\a||\b|}\mathbb{DE^\dagger}_{v\boxtimes \a}\circ \mathbb{DE^\dagger}_{w\boxtimes \b}. \] Moreover, \( \mathbb{DE^\dagger} \) is local and descends to an \( \R \)-algebra action \( [\mathbb{DE^\dagger}] \) on \( \U(E) \), up to sign. Moreover, \( \mathbb{DE^\dagger} \) and \( [\mathbb{DE^\dagger}] \) are \( C^\i(M) \)-linear if the codomain has scalar multiplication given by \( R_f \).
\end{prop}

As above, specializing to the case that \( E=T^*M \), the trace \( tr(\mathbb{DE^\dagger}) \) is well-defined and descends via \( \Phi \) to \( tr([\mathbb{DE^\dagger}]) \). Identifying \( \U(T^*M) \) with \( \U(M) \) using the metric, the map \( tr([\mathbb{DE^\dagger}]) \) is dual to exterior derivative (i.e. it is the boundary map \( \p \).)

Alternatively, we can think of \( \mathbb{E}^\dagger \) as an action of \( C^\i(\wedge(E^*)) \) on \( \otimes(TM)\boxtimes \wedge(E) \), defining \[ \mathbb{E}_\theta^\dagger(\a_1\wedge\dots\wedge \a_k) = \sum_{L_r\subset (1,\dots,k)} (-1)^{l_1+\cdots+l_r+r(r+1)/2}\theta_p(\a_{L_r}) \cdot \, \a_1\wedge\cdots\wedge\widehat{\a_{L_r}}\wedge\cdots\wedge \a_k \] and extending to the rest of \( \otimes(TM)\boxtimes \wedge(E) \) as in \eqref{exdagger}. Thus, one can think of \(  \mathbb{E}^\dagger \) as an extension of the Koszul differential. Using this definition and specializing to \( E=TM \), we have \[ tr([\mathbb{DE^\dagger}])= [\mathbb{D}_{e_i}]\circ [\mathbb{E}_{e^i}^\dagger] =\p. \] That is, \( tr([\mathbb{DE^\dagger}]) \) is the boundary operator on finitely supported currents. In a local frame \( (e_i) \) about \( p \), using the higher order Christoffel symbols and the commutation relation \eqref{eq:comm2},
\begin{align*}
		tr(\mathbb{DE^\dagger})(e_I\boxtimes e_1\wedge\cdots\wedge e_k) &= \mathbb{D}_{e_i}\circ \mathbb{E}_{e_i}^\dagger (e_I\boxtimes e_1\wedge\cdots\wedge e_k)\\
		&=(-1)^{j+1} \G_{I_{(1)},i}^r(p) \G_{I_{(2)},i}^s(p)<e_r,e_j>_p  e_{I_{(3)}}\otimes e_s \boxtimes e_1\wedge\cdots\wedge\hat{e}_j\wedge\cdots\wedge e_k.
\end{align*}

Using this formula, we note that \( (tr(\mathbb{DE^\dagger}))^2\neq 0 \) hence is not a differential, however:

\begin{prop}
	The map \( tr(\mathbb{DE^\dagger}) \) commutes with \( \D_\otimes \) (see Definition \ref{def:coprodbox},) its image is contained in \( ker(\Phi_p) \), and \[ tr(\mathbb{DE^\dagger})(\otimes^{\leq r}(TM)\boxtimes \wedge^k(E))\subseteq \otimes^{\leq r+1}(TM)\boxtimes \wedge^{k-1}(E). \]
\end{prop}

\addcontentsline{toc}{section}{References} 
\bibliography{Harrisonbib.bib}{}

@book{schwartz0,
	Author = {Schwartz, Laurent},
	Publisher = {Publications de l'Institut de Math{\'e}matique de l'Universit{\'e} de Strasbour},
	Title = {Theorie des Distributions},
	Year = {1945}}

@book{1}

@misc{coprodlie,
	title={The Hopf Algebraic Structure of Finitely Supported Currents on a Lie Group},
	author={Pugh, Harrison H.},
	year={2026}}

@misc{melrose,
      title={A remark on distributions and the de Rham theorem}, 
      author={Richard B. Melrose},
      year={2011},
      eprint={1105.2597},
      archivePrefix={arXiv},
      primaryClass={math.DG},
      url={https://arxiv.org/abs/1105.2597}
}
\bibliographystyle{amsalpha}
\end{document}